\documentclass[onefignum,onetabnum]{siamart251216}

\usepackage[T1]{fontenc}
\usepackage{lmodern}
\usepackage{microtype}

\usepackage{amssymb}
\usepackage{bm}              
\usepackage{mathtools}       
\usepackage{physics}         
\usepackage{cases}           
\usepackage{empheq}          
\usepackage{tensor}          
\usepackage{ntheorem}
\newtheorem{remark}{Remark}

\usepackage{subcaption}      
\usepackage{float}
\usepackage{placeins}        
\usepackage{booktabs}        
\usepackage{multirow}
\usepackage{multicol}
\usepackage{array}
\usepackage{longtable}
\usepackage{tabularx}
\usepackage{colortbl,xcolor}

\usepackage{algorithmicx}
\usepackage{algpseudocode}
\usepackage{listings}
\usepackage{enumitem}
\usepackage{tikz}
\usetikzlibrary{shapes.geometric, arrows.meta}
\usepackage{pgfplots}
\pgfplotsset{compat=1.18}

\crefalias{figure.caption}{figure}
\crefalias{table.caption}{table}

\newsiamthm{property}{Property}
\newsiamthm{assumption}{Assumption}
\newsiamthm{derivation}{Derivation}

\newsiamremark{example}{Example}
\newsiamremark{question}{Question}
\newsiamremark{answer}{Answer}

\crefname{property}{Property}{Properties}
\crefrangelabelformat{property}{#3#1#4--#5#2#6}
\crefname{derivation}{Derivation}{Derivations}
\crefrangelabelformat{derivation}{#3#1#4--#5#2#6}
\crefname{assumption}{Assumption}{Assumptions}
\crefrangelabelformat{assumption}{#3#1#4--#5#2#6}



\newcommand{\pt}{\partial_t}
\newcommand{\px}{\partial_x}
\newcommand{\pvv}{\partial_v}

\newcommand{\Dt}{\Delta t}

\newcommand{\mb}[1]{\boldsymbol{#1}}

\newcommand{\Phil}{\Phi^L}

\newcommand{\Iz}{\mathcal{I}_z}

\newcommand{\LL}{\textcolor{black}}
\headers{SL Bi-Fidelity for uncertain VP}{X. Cai, J. Guo, and L. Liu}

\title{A Structure-Preserving Semi-Lagrangian based bi-fidelity method for the Vlasov--Poisson system with uncertainties
\thanks{Submitted to the editors October 16, 2026.
\funding{
The work of the third author was supported by the National Key R\&D Program of China (2021YFA1001200), the Ministry of Science and Technology of China, the General Research Fund of the Research Grants Council of Hong Kong (14303022, 14301423, 14307125, 14310726).
The work of the first author and the second author was partially supported by the NSFC (No 12671463), Guangdong basic and applied basic research foundation [2025A1515012182], Guangdong and Hong Kong Universities “1+1+1” Joint Research Collaboration Scheme [2025A0505000014].
}}}

\author{
  Xiaofeng Cai\thanks{Department of Mathematics, Faculty of Arts and Sciences, Beijing Normal University, Zhuhai,
519087, P.R. China (\email{xfcai@bnu.edu.cn}).}
  \and
  Jiaojiao Guo\thanks{Department of Mathematics, Faculty of Arts and Sciences, Beijing Normal University, Zhuhai,
519087, P.R. China (\email{gjiaojiao@mail.bnu.edu.cn}).}
  \and
  Liu Liu\thanks{The Chinese University of Hong Kong, Hong Kong (\email{lliu@math.cuhk.edu.hk}).}
}

\begin{document}

\maketitle

\begin{abstract}
Uncertainty quantification for the Vlasov–Poisson (VP) systems entails repeated dynamic simulations, incurring immense computational costs.
We develop a structure-preserving Semi-Lagrangian (SL) bi‑fidelity approximation of macroscopic observables.
The MUSCL is used as a low-fidelity solver for pressureless Euler--Poisson system, allowing for a more cost-effective exploration of the stochastic parameter space.
A conservative SL scheme combined with high-order discontinuous Galerkin is used as a high-fidelity solver for collisionless VP system, whose mass conservation, positivity preservation, and low numerical dissipation provide a more reliable high-fidelity snapshots for bi-fidelity reconstruction.
We further derive a conditional
error bound that separates projection, discretization, and model errors.
Numerical experiments for both deterministic and uncertain VP problems, including multispecies and higher-dimensional settings, demonstrate that the proposed framework achieves accurate and robust uncertainty propagation.
\end{abstract}



\section{Introduction}
The Vlasov–Poisson (VP) system is a fundamental kinetic model for describing collisionless plasmas and self-gravitating systems. In contrast to macroscopic fluid models, it evolves the distribution function \(f(\mathbf{x},\mathbf{v},t)\) in phase space and is therefore capable of capturing important kinetic phenomena, including wave--particle interactions, collisionless Landau damping, and the formation of fine-scale structures through phase-space filamentation \cite{Bedrossian2022544, Huang2025572}.
These features, however, come at a substantial computational cost. In particular, the phase-space dimension increases rapidly with the spatial and velocity dimensions, making direct numerical simulation increasingly expensive for multidimensional problems.
The Particle-in-Cell (PIC) methods \cite{CAMPOREALE2016,KORMANN2021} are widely used for kinetic plasma high-dimensional simulations, but their particle-based nature introduces statistical noise and may require a large number of particles to obtain accurate macroscopic quantities.
Eulerian approaches \cite{Banks2019,CAI2022SIAM,Filbet2001394,Manzini2017,YE2024} avoid this statistical noise by directly discretizing the distribution function, but their computational cost grows significantly with the dimension of phase space.

Semi-Lagrangian (SL) methods \cite{SONNENDRUCKER1999201} provide an attractive approach for reducing the computational cost of Eulerian simulations. SL methods are not subject to the conventional Courant--Friedrichs--Lewy (CFL) restriction on the time step. This allows relatively large time steps while retaining a fully Eulerian representation of the distribution function. 
High-order reconstruction techniques, such as discontinuous Galerkin (DG) and finite-volume (FV) methods, can further improve the resolution of fine phase-space structures and reduce numerical diffusion \cite{CAI2021110036,XIONG2014,ZHENG2022114973}.
In addition, preserving fundamental physical properties is also important, since excessive numerical dissipation or violation of conservation laws may significantly affect the long-time behavior of the computed solution.
High-order conservative SL discretizations therefore provide some appealing frameworks for accurate and physically reliable kinetic simulations \cite{CHEN2025,CROUSEILLES20101927,Liu2023}.
Current research also focuses on extending SL method to genuinely high-dimensional settings.
To this end, dimensional-reduction and data-compression techniques, including adaptive mesh, sparse-grid approximations, and low-rank tensor representations, have increasingly been incorporated into SL frameworks \cite{Bokanowski2013,Kormann2015,ZHENG2025113970}. 
These approaches provide a promising route toward practical simulations of high-dimensional problems.

While classic VP models describe idealized deterministic dynamics, physical parameters, initial conditions, or external fields are not known exactly  in many applications, which naturally modeled as random variables,  and motivates the development of uncertainty quantification (UQ) methods for kinetic equations \cite{Pareschi2021}.
The computational challenge becomes substantially more severe when uncertainties are incorporated into the VP system.
Classical approaches to UQ include stochastic spectral methods such as polynomial chaos and stochastic Galerkin formulations, stochastic collocation, and sampling-based techniques such as Monte Carlo, quasi-Monte Carlo, etc.\cite{Giles2008,JIN201715,Syed2024,XIAO2021110139,XiuD200722,XIU200527}.
Despite their generality, these methods can become prohibitively expensive when each deterministic realization itself requires a high-resolution kinetic computation, particularly in high-dimensional random spaces or when a large number of samples or stochastic degrees of freedom is required.
This motivates the development of efficient reduced-cost approaches that can retain the essential kinetic information while substantially reducing the number of expensive VP simulations.

Bi-fidelity and, more generally, multi-fidelity methods provide a promising strategy for high-dimensional UQ, with a series of studies in \cite{LIN202547, Liu2020, LIU2026115076, Narayan2014, Peherstorfer2018, Zhu2017}. 
The key idea is to use the low-fidelity model to explore the parameter space and identify representative parameter samples, at which the high-fidelity model is then evaluated.
The resulting high-fidelity snapshots are used to construct a reduced approximation over the stochastic parameter domain. The efficiency of this approach relies on the existence of a sufficiently strong correlation between the high-fidelity and low-fidelity models.
At the same time, the quality of the high-fidelity snapshots is crucial: numerical errors, excessive dissipation, or violations of physical constraints in the high-fidelity solver can be inherited by the reduced representation and consequently affect the bi-fidelity approximation over the entire parameter domain.
For kinetic equations, where fine-scale structures and positivity are particularly important, an accurate and structure-preserving high-fidelity solver is therefore an essential component of an effective bi-fidelity framework.

Based on the analysis above, we develop a structure-preserving SL bi‑fidelity approximation of uncertain VP systems based on conservative SL (CSL) snapshots. The main contributions of this work can be summarized as follows:
\begin{itemize}
    \item[(i)] \textbf{CSL high-fidelity solver.}
    We develop Strang-splitting CSL method with high-order DG reconstruction as the high-fidelity solver for the uncertain VP system. The proposed solver combines large time-step efficiency, high-resolution phase-space reconstruction, positivity-preserving, and mass conservation, providing reliable high-fidelity snapshots for uncertainty quantification.
    \item[(ii)] \textbf{An error analysis.}
    We exploit the pressureless Euler--Poisson (EP) system as a low-fidelity surrogate that retains the essential macroscopic behavior of the VP system while eliminating the expensive velocity-space discretization. Based on this kinetic–macroscopic correspondence, we develop a bi-fidelity approximation framework and establish an error analysis.
    \item[(iii)] \textbf{Numerical simulation.}
    We demonstrate the applicability of the proposed method to challenging kinetic problems, including multispecies configurations and higher-dimensional settings with complex phase-space structures.
\end{itemize}

The remainder of this paper is organized as follows.
Section~2 introduces the VP system and its corresponding pressureless EP model.
Section~3 presents a bi-fidelity framework based on a Strang-splitting structure-preserving SL high-fidelity solver and a second-order MUSCL low-fidelity solver, with explanations of positivity and conservation also included.
Section~4 develops the error analysis for the bi-fidelity approximation.
Numerical results are presented in Section~5, and concluding remarks and directions for future works in Section~6.
\section{Microscopic and Macroscopic Models}
This section introduces the VP system and the pressureless EP system, which serve as the high- and low-fidelity models, respectively. 
We summarize their relevant structural, conservation properties, and discuss the connection between the microscopic and macroscopic descriptions. 

\subsection{Vlasov--Poisson System}
Let \(f(\mathbf{x},\mathbf{v}, t)\) denote the distribution function, which quantifies the probability of detecting a particle at velocity \(\mathbf{v}\) at position \(\mathbf{x}\) at time \(t\). 
In this work, we based it on a prototype of a collision-free single-species VP system controlled electronically:
\begin{align}\label{eq:single_VP}
\begin{cases}
\partial_t f + \mathbf{v}\cdot\nabla_\mathbf{x} f + E(\mathbf{x}, t)\cdot\nabla_\mathbf{v} f = 0, \qquad (\mathbf{x},\mathbf{v})\in\Omega_\mathbf{x}\times\Omega_\mathbf{v} \subset \mathbb{R}^d \times \mathbb{R}^d,\\[2mm]
E(\mathbf{x}, t) = -\nabla_\mathbf{x}\phi(\mathbf{x}, t), \quad  -\varepsilon \Delta_\mathbf{x}\phi(\mathbf{x}, t) = \rho(\mathbf{x}, t) - \rho_0,
\end{cases}
\end{align}
where \(d \in \{1, 2, 3\}\) denotes dimension,
the spatial domain \(\Omega_\mathbf{x}\) is assumed to be periodic, and the velocity domain \(\Omega_\mathbf{v}\) is chosen sufficiently large such that the distribution function \(f\) decays to zero (or compactly supported) on the boundary \(\partial\Omega_\mathbf{v}\).
The self-consistent electrostatic potential \(\phi(\mathbf{x},t)\) determines the electric field \(E(\mathbf{x},t)\).
The charge density $\rho(\mathbf{x},t) = \int_{\Omega_\mathbf{v}} f(\mathbf{x},\mathbf{v},t)\,d\mathbf{v}$ and $\rho_0=\frac{1}{\Omega_\mathbf{x}}\int_{\Omega_\mathbf{x}}\int_{\Omega_\mathbf{v}} f(\mathbf{x},\mathbf{v},0)\,d\mathbf{v} d\mathbf{x}.$
The \(\varepsilon=\left(\frac{\lambda_D}{L}\right)^2\) is dimensionless parameter characterized by the normalized Debye length, our work discusses the case \(\varepsilon=1\). 

The VP system \eqref{eq:single_VP} has the following conservation laws,
these conservation laws constitute the fundamental bedrock for physical analysis and serve as critical benchmarks for verifying the fidelity of numerical schemes:
    \begin{align*}
    &\frac{d}{dt}\int_{\Omega_\mathbf{x}}\rho(\mathbf{x},t)\,d\mathbf{x} = \frac{d}{dt}\int_{\Omega_\mathbf{x}}\int_{\Omega_\mathbf{v}} f(\mathbf{x},\mathbf{v},t)\,d\mathbf{v}d\mathbf{x} = 0\quad\text{(Mass)},\\
    &\frac{d}{dt}\int_{\Omega_\mathbf{x}}\int_{\Omega_\mathbf{v}} \mathbf{v} f(\mathbf{x},\mathbf{v},t)\,d\mathbf{v}d\mathbf{x} = 0\quad\text{(Momentum)},\\
    &\frac{d}{dt}\left( \frac{1}{2}\int_{\Omega_\mathbf{x}}\int_{\Omega_\mathbf{v}} |\mathbf{v}|^2 f(\mathbf{x},\mathbf{v},t)\,d\mathbf{v}d\mathbf{x} + \frac{1}{2}\int_{\Omega_\mathbf{x}} |E(\mathbf{x},t)|^2\,d\mathbf{x} \right) = 0\quad\text{(Energy)}.
    \end{align*}

Uncertainties typically originate from empirical errors in measuring equilibrium density, velocity, and temperature profiles, fluctuations in external magnetic/electric fields, and stochastic perturbations in initial distribution modes. In order to depict these phenomena, we introduce a \(d_z\)-dimensional random parameter \(z\in\mathcal{I}_z = [-1,1]^{d_z}\) with a given probability distribution. \((\mathbf{x},\mathbf{v},z) \in \Omega_\mathbf{x} \times \Omega_\mathbf{v} \times \Iz \subset \mathbb{R}^d \times \mathbb{R}^d \times \mathbb{R}^{d_z}\), the VP system \eqref{eq:single_VP} then becomes the following stochastic system
\begin{subequations}
\begin{align}
\label{eq:VP_random}
  &\pt f + \mathbf{v}\cdot\nabla_\mathbf{x} f + E\cdot \nabla_\mathbf{v} f = 0,
  \tag{Vlasov}\\
  \label{eq:Poisson}
& E = -\nabla_\mathbf{x}\phi, \quad -\Delta_\mathbf{x} \phi = \rho - \rho_0, \quad
\rho(\mathbf{x},t,z) = \int_{\Omega_\mathbf{v}} f(\mathbf{x},\mathbf{v},t,z) \dd \mathbf{v},
  \tag{Poisson}
\end{align}
\end{subequations}
where the initial conditions for \(f\) are prescribed as random fields parameterized by \(z\), \(f(\mathbf{x}, \mathbf{v}, 0, z) = f_{\mathrm{in}}(\mathbf{x}, \mathbf{v}, z)\) and $\rho_0(\cdot, z)=\frac{1}{\Omega_\mathbf{x}}\int_{\Omega_\mathbf{x}}\int_{\Omega_\mathbf{v}} f(\mathbf{x}, \mathbf{v}, 0, z)\,d\mathbf{v} d\mathbf{x}.$
The random parameter \(z\) can represent, for example, uncertain amplitudes, phases, wavenumber, or background profiles.

\subsection{The macroscopic fluid equations}
Define the macroscopic moment vector by \( \mathbf{W} = (\rho, \rho \mathbf{u}, \mathcal{E})^\top
  = (\int_{\Omega_\mathbf{v}} f \dd \mathbf{v}, \int_{\Omega_\mathbf{v}} \mathbf{v}f \dd \mathbf{v}, \int_{\Omega_\mathbf{v}} \frac{|\mathbf{v}|^2}{2}f \dd \mathbf{v})^\top\),
and the mean velocity \( \mathbf{u} = \frac{1}{\rho} \int_{\Omega_\mathbf{v}} \mathbf{v}f \dd \mathbf{v}.\)
Multiplying the Vlasov equation by the moments vector
$\mb{m}(\mathbf{v}) = (1,\, \mathbf{v},\, \mathbf{v}^2/2)^\top$
and integrating over $\mathbf{v} \in \Omega_\mathbf{v}$, using integration by parts for the force term (assuming $f \to 0$ as
$|\mathbf{v}|\to\infty$), the mass, momentum, and kinetic-energy balances, one obtains the unclosed moment equations
\begin{equation}
  \label{eq:moment_unclosed}
  \pt \mathbf{W} + \nabla_\mathbf{x} \cdot \mathcal{F}(f)
  =
  \begin{pmatrix}0 \\ \rho E \\ \rho \mathbf{u} \cdot E\end{pmatrix},
  \qquad
  \mathcal{F}(f)
  = \begin{pmatrix}\int_{\Omega_\mathbf{v}} \mathbf{v}f \dd \mathbf{v} \\[2pt]
                   \int_{\Omega_\mathbf{v}} \mathbf{v} \otimes \mathbf{v} f \dd \mathbf{v} \\[2pt]
                   \int_{\Omega_\mathbf{v}} \frac{|\mathbf{v}|^2 \mathbf{v}}{2}f \dd \mathbf{v}\end{pmatrix}.
\end{equation}
These equations are exact consequences of the kinetic model but is {\it not a closed} macroscopic system. Indeed, the momentum equation involves the pressure tensor \((\mathbb P = \int_{\Omega_v} (\mathbf{v} - \mathbf{u}) \otimes (\mathbf{v} - \mathbf{u}) f \, d\mathbf{v})\), while the energy flux contains still higher-order velocity moments. A closure assumption is therefore required to express the macroscopic dynamics solely in terms of a finite number of moments.

Now, we introduce the cold-electron assumption \cite{Brenier01012000}: the eletrons velocity distribution is highly concentrated initially and remains concentrated under the evolution of collisionless Vlasov dynamics. 
Let \(\{f^\varepsilon(\mathbf{x},\mathbf{v}, t)\}_{\varepsilon>0}\) be a family of general solutions to the VP system.
In the limit \(\varepsilon \to 0\), the distribution \(f^\epsilon\) converges weakly to Dirac \(\delta\) function,
all components of the moment flux \(\mathcal{F}(f)\) can be precisely expressed in terms of the moments:
\begin{align*}
\int_{\Omega_\mathbf{v}} f \mathbf{v} \, d\mathbf{v} = \rho \mathbf{u},\quad
\int_{\Omega_\mathbf{v}} \mathbf{v} \otimes \mathbf{v} f \dd \mathbf{v}=\rho \mathbf{u} \otimes \mathbf{u},\quad
 \mathcal{E}:=\int_{\Omega_\mathbf{v}} \frac{|\mathbf{v}|^2}{2} f \, d\mathbf{v}  = \frac{1}{2} \rho |\mathbf{u}|^2, 
\end{align*}
meanwhile, the temperature is automatically satisfied as \(
  T = \frac{1}{d}\left(\frac{2\mathcal{E}}{\rho} - |\mathbf{u}|^2 \right)= 0,
\)
which yields the pressure \(\mathbb P = \rho T I_d\equiv 0\). 

Let \((\mathbf{x},z) \in \Omega_\mathbf{x} \times \Iz \subset \mathbb{R}^d \times \mathbb{R}^{d_z}\), substituting the above monokinetic closure into the unclosed moment equations, the energy is determined by \(\mathcal E=\tfrac12\rho|\mathbf u|^2\), it is no longer an independent unknown but is instead derived from the mass and momentum equations.
The fully closed pressureless EP system \cite{Brenier01012000, Han-Kwan01082011} with random parameters:
\begin{align}
  \pt \rho + \nabla_\mathbf{x}\cdot(\rho \mathbf{u}) &= 0, \label{eq:EP_mass} \tag{Mass}\\
  \pt(\rho \mathbf{u}) + \nabla_\mathbf{x}\cdot(\rho \mathbf{u} \otimes \mathbf{u} ) &= \rho E, \label{eq:EP_mom} \tag{Momentum}
\end{align}
coupled with the same Poisson equation \eqref{eq:Poisson} governing the self-consistent electric field.
This pressureless EP model retains the evolution of the leading moments and the electrostatic coupling of the VP system, while eliminating the velocity-space discretization and thereby substantially reducing the computational cost.
These provide the modeling foundation for the bi-fidelity construction and the subsequent error analysis.
\section{A structure-preserving SL bi-fidelity method}
It is worth noting that for uncertain problems, we take the VP model as the high-fidelity model, and the macroscopic quantity  \( \mathbf{U}^H(z) = \bigl(\rho^H,\; u^H\bigr)^\top\) obtained by moment calculation of \(f\) as the high-fidelity solution, and we take the pressureless EP model as the low-fidelity model, with its solution denoted as \(
  \mathbf{U}^L(z) = \bigl(\rho^L,\; u^L\bigr)^\top\).

For the convenience of introducing numerical methods and maintaining consistency in subsequent theoretical analysis, we will introduce the algorithm for one-dimensional space and one-dimensional velocity (1D1V).

\subsection{High-fidelity solver}
\label{high-fidelity}
The CSL method combined with high-order spatial discretization \cite{CAI2017, CAI20220301, CHO2021110281, CROUSEILLES20101927} can provide a high-resolution, structure-preserving solver for bi-fidelity surrogate frameworks to evaluate representative high-fidelity snapshots, ensuring the accurate capture of the underlying low-dimensional solution manifold, thereby improving algorithm accuracy and the physical consistency of reconstructed approximations in parameter space.
We introduce a conservative semi-Lagrangian discontinuous Galerkin (CSLDG) method that can eliminate the CFL restriction while maintaining positivity and strict mass conservation for solving deterministic VP model.

The VP equations \eqref{eq:single_VP} couple advection in $x$ (at speed $v$) with acceleration in $v$ (at electric field $ E$).
Let the time interval \([0,T]\), the time step \(\Dt = T/N_t\), and \(t^n=n\Dt, n=0 \cdots N_t\),
denote the numerical solution of \(f(x,v,t)\) at time \(t^n\) as \(f^n\).
A Strang splitting is second-order accurate,
which decomposes one full time step $\Dt$ into several consecutive sub-steps. One full time step advances the solution as
\begin{equation*}
 f^n \;\xrightarrow {\mathcal{A}_{x}^{\Dt/2}\left(f^n\right)} f^*
  \;\xrightarrow{\;\text{Poisson: } \partial_x E = \rho - \rho_0\;}
  E^* \;\xrightarrow{\mathcal{B}_{v}^{\Dt}(f^*,E^*)} f^{**}
  \;\xrightarrow{\mathcal{A}_{x}^{\Dt/2}\left(f^{**}\right)} f^{n+1}
\end{equation*}
where each operator solves:
\begin{itemize}[leftmargin=2em]
  \item $\mathcal{A}_{x}^{\Dt/2}$: This operator describes pure advection in the spatial direction\(x\) of phase space. The distribution function is transported with the velocity \(v\). After the spatial-advection substep, the charge density is updated and the electric field is obtained from the Poisson equation: 
  \[\pt f + v\px f = 0, \quad \partial_x E(x, t) = \rho - \rho_0.\quad \]
  \item $\mathcal{B}_{ v}^{\Dt}$: This operator describes advection in the velocity direction induced by the electric field. During this substep, the charge density is preserved, and hence the electric field determined by the Poisson equation remains constant. Therefore, \(E(x, t)\) acts as a fixed advection velocity in velocity space:
  \[\pt f + E(x, t)\pvv f = 0, \quad \pt E(x, t)=0.\]
\end{itemize}
On the periodic domain $[0,L]$, the Poisson equation
\eqref{eq:Poisson} is solved directly using Fourier pseudospectral methods (such as, the Fast Fourier Transform (FFT), the calculation details see \cite{SHEN2011}).
Define the charge-source $S = \rho - \rho_0$.
Its Fourier coefficients $\hat{S}_k = \mathcal{F}[S]_k$ satisfy
\(
  \hat{E}_k = \frac{-i}{\kappa_k}\,\hat{S}_k,  k \neq 0,
 \hat{E}_0 = 0,
\)
where $\kappa_k = 2\pi k/L$.
The electric field is $E(x, t) = \mathrm{Re}[\mathcal{F}^{-1}[\hat{E}](x, t)]$.

\subsubsection{CSLDG method}

We now detail the DG spatial discretization of the CSL framework.
The phase-space domain $\Omega = \Omega_x \times \Omega_v = [x_L, x_R] \times [v_L, v_R]$ is partitioned into $N_x \times N_v$ Cartesian cells $\mathcal{T}_h = \{ I_i \times J_j \}$, where
\begin{align*}
    I_i &= [x_{i-1/2}, x_{i+1/2}], \quad x_{i+1/2} = x_L + i\Delta x, \quad \Delta x = \frac{x_R - x_L}{N_x}, \quad i = 1, \dots, N_x, \\
    J_j &= [v_{j-1/2}, v_{j+1/2}], \quad v_{j+1/2} = v_L + j\Delta v, \quad \Delta v = \frac{v_R - v_L}{N_v}, \quad j = 1, \dots, N_v.
\end{align*}
The cell centers are $\bar{x}_i = (x_{i-1/2}+x_{i+1/2})/2$ and $\bar{v}_j = (v_{j-1/2}+v_{j+1/2})/2$.
We define the discontinuous finite element space
\(
    V_h^p = \left\{ w(x,v) : \left. w \right|_{I_i \times J_j} \in P^p(I_i) \otimes P^p(J_j), \; \forall i, j \right\},
\)
with $P^p$ the space of polynomials of degree at most $p$ in each variable. On each element $I_i \times J_j$, the distribution function is approximated as
\begin{equation}
    \left. f(x,v,t) \right|_{I_i \times J_j} \approx \sum_{k=0}^p \sum_{l=0}^p U_{ij}^{kl}(t) \phi_k(\xi_x) \phi_l(\xi_v),
    \label{eq:expansion_2d}
\end{equation}
where the local reference coordinates are
\(
    \xi_x = \frac{x - \bar{x}_i}{\Delta x} \in \left[ -\frac{1}{2}, \frac{1}{2} \right], \xi_v = \frac{v - \bar{v}_j}{\Delta v} \in \left[ -\frac{1}{2}, \frac{1}{2} \right].
\)
The scaled Legendre polynomials $\{\phi_k\}_{k=0}^p$ on $[-1/2,1/2]$ satisfy
\(
    \int_{-1/2}^{1/2} \phi_k(\xi) \phi_l(\xi) \, d\xi = \frac{1}{2k+1} \delta_{kl} = \frac{1}{a_k} \delta_{kl}, a_k = 2k+1.
\)
The initial condition is projected via $L^2$-projection onto $V_h^p$. For each cell, the initial coefficients are
\begin{equation}
    U_{ij}^{kl}(0) = a_k a_l \int_{-1/2}^{1/2} \int_{-1/2}^{1/2} f^0\left(\bar{x}_i + \xi_x \Delta x, \, \bar{v}_j + \xi_v \Delta v\right) \phi_k(\xi_x) \phi_l(\xi_v) \, d\xi_x d\xi_v,
    \label{eq:initial_projection}
\end{equation}
evaluated with a tensor-product $(p+2)$-point Gauss–Legendre quadrature for exact integration of polynomials up to degree $p$.

\textbf{$x$-Direction CSLDG Update $\mathcal{A}_{x,h}^{\Dt/2}$:}
In the operator splitting framework, the $x$-advection sub-step decouples across each velocity ordinate $v$. For a fixed velocity node, the governing equation reduces to the 1D linear advection equation:
\begin{equation}
    \partial_t f + v \partial_x f = 0, \quad x \in \Omega_x, \quad t \in [t^n, t^*],
    \label{eq:advection_1d}
\end{equation}
with $t^* = t^n + \Delta t^* = t^n + \frac{\Delta t}{2}$.
To obtain a mass-conservative weak form, we introduce a time-dependent test function $\psi(x,t)$ satisfying the adjoint problem:
\begin{equation}
    \begin{cases}
        \partial_t \psi + v \partial_x \psi = 0, & t \in [t^n, t^*], \\
        \psi(x, t^*) = \Psi(x) \in P^p(I_i).
    \end{cases}
    \label{eq:adjoint_problem}
\end{equation}
The characteristics of Eq.~\eqref{eq:advection_1d} are given by $dx/dt = v$. Tracing a point $x \in I_i$ at $t^*$ backward to $t^n$ yields $x(t^n) = x - v \Delta t^*$. Thus $\psi$ is invariant along characteristics:
\begin{equation*}
    \psi(x, t) = \Psi(x + v(t^* - t)) \implies \psi(x, t^n) = \Psi(x + v \Delta t^*).
\end{equation*}
Let $I_i(t) = [x_{i-1/2}(t), x_{i+1/2}(t)]$ be the moving Lagrangian interval with endpoints moving at speed $v$, and $I_i(t^*) = I_i$, the upstream interval at $t^n$ is
\begin{equation*}
    I_i^* \equiv I_i(t^n) = [x_{i-1/2}^*, x_{i+1/2}^*] = [x_{i-1/2} - v\Delta t^*, \; x_{i+1/2} - v\Delta t^*].
\end{equation*}
Applying Reynolds transport theorem to $f\psi$ over $I_i(t)$ and using Eqs.~\eqref{eq:advection_1d} and \eqref{eq:adjoint_problem}, we get
\(
    \frac{d}{dt} \int_{I_i(t)} f(x,t)\psi(x,t) \, dx = 0.
    \)
Integrating over $[t^n, t^*]$ yields the fundamental CSLDG variational equality:
\begin{equation*}
    \int_{I_i} f(x,t^*) \Psi(x) \, dx = \int_{I_i^*} f(x,t^n) \psi(x,t^n) \, dx, \quad \forall \Psi \in P^p(I_i).
    \label{eq:csl_weak_form}
\end{equation*}
Choosing $\Psi(x) = \phi_k((x-\bar{x}_i)/\Delta x)$, the updated modal coefficients are
\begin{equation}
    U_i^{k,*} = \frac{a_k}{\Delta x} \int_{I_i^*} f^n(x) \phi_k\left( \frac{x + v \Delta t^* - \bar{x}_i}{\Delta x} \right) dx, \quad k = 0, \dots, p.
    \label{eq:csl_moments}
\end{equation}
For $k=0$, this recovers \textbf{mass conservation}: $\int_{I_i} f^*(x) dx = \int_{I_i^*} f^n(x) dx$.

For arbitrary CFL number $\nu = |v|\Delta t^*/\Delta x > 1$, the upstream interval $I_i^*$ may overlap multiple Eulerian cells. We decompose $I_i^*$ along the background grid:
\(
    I_i^* = \bigcup_{\ell=1}^{M_i} I_{i,\ell}^*,
\)
where each $I_{i,\ell}^* = I_i^* \cap I_{j_\ell}$ lies strictly inside a single Eulerian cell $I_{j_\ell}$. Then Eq.~\eqref{eq:csl_moments} becomes the piecewise sum
\begin{equation}
    U_i^{k,*} = \frac{a_k}{\Delta x} \sum_{\ell=1}^{M_i} \int_{I_{i,\ell}^*} f^n(x) \phi_k\left( \frac{x + v \Delta t^* - \bar{x}_i}{\Delta x} \right) dx.
    \label{eq:subdivided_integral}
\end{equation}
Summing the quadrature contributions over $\ell$ gives the updated coefficients $U_i^{k,*}$ for each target cell $i$ and velocity $v$, completing the $x$-half-step. The $v$-direction update and the remaining time step are performed analogously.

\begin{remark}
To evaluate \eqref{eq:subdivided_integral} in practice, we first trace the upstream endpoints and determine the index span $[j_{\min}, j_{\max}]$ via
{\small\[
    j_{\min} = \left\lfloor \frac{x_{\min}^* - x_L}{\Delta x} \right\rfloor + 1, \quad j_{\max} = \left\lfloor \frac{x_{\max}^* - x_L}{\Delta x} \right\rfloor + 1,
\]}
where $x_{\min}^* = \min(x_{i-1/2}^*, x_{i+1/2}^*)$, $x_{\max}^* = \max(x_{i-1/2}^*, x_{i+1/2}^*)$, and \(\left\lfloor \cdot \right\rfloor \) represents rounding down.
This yields all sub-intervals $I_{i,\ell}^* = [A_\ell, B_\ell]$.
On each sub-interval, $f^n$ is represented using the DG coefficients of the host cell $I_{j_\ell}$:
{\small\[
    \left. f^n(x) \right|_{I_{i,\ell}^*} = \sum_{m=0}^p U_{j_\ell}^{m,n} \phi_m\left(\frac{x - \bar{x}_{j_\ell}}{\Delta x}\right).
\]}
Then Gauss–Legendre quadrature with $Q \ge \lceil (2p+1)/2 \rceil$ points is applied. Mapping $[A_\ell, B_\ell]$ to $\xi \in [-1,1]$ via $x(\xi) = (A_\ell+B_\ell)/2 + \xi (B_\ell-A_\ell)/2$, each sub-integral is approximated by
{\small\begin{equation*}
    \int_{I_{i,\ell}^*} f^n(x) \psi(x,t^n) \, dx \approx \frac{B_\ell - A_\ell}{2} \sum_{q=1}^Q w_q \, f^n(x(\xi_q)) \, \phi_k\left( \frac{x(\xi_q) + v\Delta t^* - \bar{x}_i}{\Delta x} \right).
\end{equation*}}
\end{remark}

\textbf{\(v\)-Direction CSLDG Update \(\mathcal{B}_{v,h}^{\Delta t}\):}
The \(v\)-sweep solves \(\partial_t f + E(x,t)\partial_v f = 0\) with the field frozen at its most recent value (after the \(x\) half-step and Poisson solve), approximated on each spatial cell \(I_i\) by its local mean \(\bar E_i\).

It is now the spatial cell \(i\) and \(x\)-mode \(k\) that are held fixed while \(l=0,\dots,p\) is updated. The backward characteristic gives \(v^* = v - \bar E_i \Delta t\), i.e.
\[
f^{**}(x,v) = f^{*}\bigl(x,\, v - \bar E_i \Delta t\bigr), \quad x \in I_i.
\]
Then, similar to the $x$-direction, apply the 1D CSLDG update:
\begin{equation}
    U_{ij}^{kl,**} = \frac{a_l}{\Delta v} \int_{\widetilde{J}_j}
    f_i^{*,k}(w)\,
    \phi_l\!\left(\frac{w + \bar{E}_i \Delta t - \bar{v}_j}{\Delta v}\right) \mathrm{d}w, \qquad k,l = 0,\dots,p,
    \label{eq:v_proj}
\end{equation}
where \(f_i^{*,k}(w) := \sum_{m=0}^p U_{im}^{km,*}\phi_m(\xi_v), w \in J_m,\) and \(\widetilde{J}_j = [v_{j-1/2} - \bar{E}_i\Delta t,\; v_{j+1/2} - \bar{E}_i\Delta t]\) is the upstream velocity interval. When \(|\bar E_i|\Delta t/\Delta v > 1\), \(\widetilde J_j\) is split along the background velocity grid exactly as in \eqref{eq:subdivided_integral}, then gives
\begin{equation}
    U_{ij}^{kl,**} = \frac{a_l}{\Delta v}\sum_{r=1}^{m} \frac{B_r-A_r}{2}\sum_{q=1}^Q w_q\, f_i^{*,k}(w(\xi_q))\, \phi_l\!\left(\frac{w(\xi_q) + \bar{E}_i \Delta t - \bar{v}_j}{\Delta v}\right),
    \label{eq:v_subdivided}
\end{equation}
where \(\widetilde{J}_j=\bigcup_{r=1}^m R_r\), \(R_r=[A_r,B_r] \subset J_{j_r}\), as in the \(x\)-sweep, \(l=0\) recovers local mass conservation in \(v\).

\subsubsection{Positivity-preserving limiter}
For the \eqref{eq:single_VP}, a positive initial datum guarantees that the corresponding solution remains nonnegative for all later times; this fundamental property is referred to as nonnegativity preservation. In the context of our CSLDG discretization, one can establish that the cell averages at the new time level $t^{n+1}$ are nonnegative provided the numerical solution $f^n$ at $t^n$ is nonnegative. 
To rigorously enforce this invariant in the high‐order SLDG framework, we incorporate a positivity‐preserving (PP) limiting strategy, following the ideas developed in \cite{Cai2018, QIU20118386}. Given the numerical solution $f(x,v,t^n)$, we construct modified polynomial $\tilde{f}(x,v)$ via
\[
\tilde{f}(x,v) = \theta\bigl(f(x,v,t^n) - \bar{f}\bigr) + \bar{f}, \qquad
\theta = \min \left\{ \left| \frac{m_0-\bar{f}}{m' - \bar{f}} \right|, 1 \right\},
\]
where $\bar{f}$ denotes the cell average of
\(\bar{f}=\frac{1}{\Delta x_i \Delta v_j}\int_{I_i \times I_j}f(x,v) dxdv,\) \(m_0 = 10^{-15}\) is a small positive threshold chosen to avoid floating-point underflow, and \(
m' := \min_{(x,v)\in I_i \times J_j} f(x,v)
\).

\begin{remark}
    Under periodic boundary conditions in \(x\), appropriate boundary conditions in \(v\), and the zero-mean condition for the electric field, the split operators preserve the total mass and the total momentum
{\small\begin{align*}
    M(\mathcal{S}_{\Delta t} f) = M\left(\mathcal{A}_x^{\Delta t/2} \mathcal{B}_v^{\Delta t} \mathcal{A}_x^{\Delta t/2} f\right) = M(f),\quad
M(f) = \int_{\Omega_x} \int_{\Omega_v} f(x, v) \, dv \, dx.\\
P(\mathcal{S}_{\Delta t} f) = P\left(\mathcal{A}_x^{\Delta t/2} \mathcal{B}_v^{\Delta t} \mathcal{A}_x^{\Delta t/2} f\right) = P(f),\quad P(f) = \int_{\Omega_x} \int_{\Omega_v} v f(x, v) \, dv \, dx.
\end{align*}}
At the discrete level, Let \(f_h^{n+1}\) represent the fully discrete spatiotemporal numerical solution of \(f\).
Both transport sub-steps are approximated using the above-mentioned CSLDG method, local mass conservation is satisfied for each cell. The introduction of a PP limiter does not compromise the local conservation property, and thus the fully discrete Strang splitting scheme can ensure the exact mass conservation
{\small\[
M_h(f_h^{n+1}) = M_h\left(\mathcal{A}_{x,h}^{\Delta t/2} \mathcal{B}_{v,h}^{\Delta t} \mathcal{A}_{x,h}^{\Delta t/2} f_h^n\right)=M_h(f_h^n), \quad M_h(f_h^n)=\sum_{i,j} f_{i,j}^n \Delta x \Delta v.
\]}
However, exact discrete momentum conservation can't be achieved directly.
For the \(\mathcal{A}_{x,h}^{\Delta t/2}\) substep, for every fixed \(v_j\), then immediately
{\small\[
P_h(f_h^*)
= \sum_j v_j \left( \sum_i f_{i,j}^* \Delta x \right) \Delta v
= P_h(f_h^n), \quad P_h(f_h^n) = \sum_{i,j} v_j f_{i,j}^n \Delta x \Delta v,
\]}
for the \(\mathcal{B}_{v,h}^{\Delta t}\) substep to obtain exact discrete momentum conservation, the conservative remapping in the velocity direction should also satisfy the first-moment relation.
Although such precise discrete momentum constraints are not imposed here, the momentum errors shown in the numerical results later are still very small, demonstrating excellent numerical conservation behavior.
\end{remark}

\subsection{Low-fidelity solver}
\label{LF}

The low-fidelity model is the pressureless EP system from the Section 2.
The compact conservation form of the system is:
\begin{equation}\label{LF_EP}
    \frac{\partial \mathbf{W}}{\partial t} + \frac{\partial \mathbf{F}(\mathbf{W})}{\partial x} = \mathbf{S}(\mathbf{W}, E),
\end{equation}
where
\begin{equation*}
    \mathbf{W} = \begin{pmatrix} \rho \\ \rho u \end{pmatrix}, \qquad
    \mathbf{F}(\mathbf{W}) = \begin{pmatrix} \rho u \\ \rho u^2 \end{pmatrix}, \qquad
    \mathbf{S}(\mathbf{W}, E) = \begin{pmatrix} 0 \\ \rho E \end{pmatrix}.
\end{equation*}
The computational domain is treated as periodic in \(x\), to discretise the system in space, we use a finite-volume method on a uniform grid with cells \(I_i = [x_{i-1/2}, x_{i+1/2}]\) of size \(\Delta x\).  
The time discretization uses forward Euler.
Let \(\bar{\mathbf{W}}_i(t) = \frac{1}{\Delta x}\int_{I_i} \mathbf{W}(x,t)\,\mathrm{d}x\) be the cell average. Integrating the conservation law over \(I_i\) gives:
\begin{equation}\label{LF_space_time}
    \bar{\mathbf{W}}_i^{n+1}
    = \bar{\mathbf{W}}_i^n
    - \frac{\Delta t}{\Delta x}\Bigl(\hat{\mathbf{F}}_{i+1/2}^n - \hat{\mathbf{F}}_{i-1/2}^n\Bigr)
    + \Delta t\,\bar{\mathbf{S}}_i^n,
\end{equation}
where \(\hat{\mathbf{F}}^n_{i+1/2}\) is the numerical flux at the interface \(x_{i+1/2}\), and \(\bar{\mathbf{S}}^n_i\) is the cell-averaged source term.
We employ the local Lax--Friedrichs flux \cite{LeVeque2002,LI2025124},
the flux is defined as:
\begin{equation}
    \hat{\mathbf{F}}^n_{i+1/2} = \frac{1}{2}\bigl[\mathbf{F}(\mathbf{W}^n_L) + \mathbf{F}(\mathbf{W}^n_R)\bigr] - \frac{\alpha_{i+1/2}}{2}\bigl(\mathbf{W}^n_R - \mathbf{W}^n_L\bigr),
\end{equation}
where the local maximum signal speed is:
\(
    \alpha_{i+1/2} = \max\!\left(|u_L|,\; |u_R|\right).
\)
Here, \(\mathbf{W}^n_L\) and \(\mathbf{W}^n_R\) are the reconstructed left and right states at the interface, we reconstruct these states from the cell averages using the MUSCL scheme introduced by van~Leer~\cite{VANLEER1979}.
The reconstruction is performed in primitive variables
\(\mathbf{q}^n=(\rho^n,u^n)^\top\):
\begin{align*}
\mathbf{q}^n_{L,i+1/2} &= \mathbf{q}^n_i + \frac12\;\mathrm{minmod}\!\bigl(\mathbf{q}^n_i - \mathbf{q}^n_{i-1},\; \mathbf{q}^n_{i+1} - \mathbf{q}^n_i\bigr),\\
\mathbf{q}^n_{R,i+1/2} 
            &= \mathbf{q}^n_{i+1} - \frac12\;\mathrm{minmod}\!\bigl(\mathbf{q}^n_{i+1} - \mathbf{q}^n_i,\; \mathbf{q}^n_{i+2} - \mathbf{q}^n_{i+1}\bigr),
\end{align*}
where the minmod function is defined as \(\mathrm{minmod}(a,b)=\frac12\bigl(\operatorname{sgn}(a)+\operatorname{sgn}(b)\bigr)\min(|a|,|b|)\).
Since the MUSCL reconstruction at cell \(i\) requires the stencil \(\{i-1, i, i+1, i+2\}\), two layers of ghost cells are appended at each end of the domain and filled by periodic wrap-around.
The electric field \(E_i^n\) is obtained from the charge density \(\rho^n_i\) by solving the Poisson equation spectrally (as described in high-fidelity solver).

\subsection{Bi-fidelity algorithm}
In order to obtain the structure-preserving SL bi-fidelity algorithm, we choose the CSLDG method with second-order Strang splitting as the high-fidelity solver to discretize the VP model and obtain the macroscopic moment \(\mathbf{U}^H\), and select the second-order MUSCL method as the low-fidelity solver to approximate the EP model and obtain the macroscopic moment \(\mathbf{U}^L\).
The bi-fidelity algorithm proceeds in two decoupled stages, here, we provide a brief review, for the specific implementation details of the bi-fidelity algorithm, please refer to \cite{Liu2020, Narayan2014, zhu201421}.

The effectiveness of the algorithm depends on the assumption that the parameterized solution manifold admits an accurate low-dimensional approximation and that the low- and high-fidelity solution spaces exhibit sufficiently correlated parametric structures. 
In a typical bi-fidelity construction, an inexpensive low-fidelity model is first evaluated over a large set of parameter samples to explore the stochastic parameter space \(\mathcal{I}_z\). A small number of representative samples are then selected, and high-fidelity simulations are performed only at these selected parameter locations. 

For a new parameter \(z\in \mathcal{I}_z\) realization, the corresponding low-fidelity solution is projected onto the reduced low-fidelity space, and the resulting expansion coefficients are used to combine the associated high-fidelity snapshots. In this manner, the bi-fidelity approximation seeks to retain the accuracy of the kinetic model while requiring only a limited number of expensive high-fidelity evaluations.
The bi-fidelity approximation $\mathbf{U}^B(z)$ is reconstructed via:
\begin{equation}
\label{eq:BF_approx}
    \mathbf{U}^B(z) = \Phi^H \mathbf{c}(z) = \Phi^H [(\Phi^L)^T \Phi^L]^{-1} (\Phi^L)^T \mathbf{U}^L(z).
\end{equation}

\begin{algorithm}[H]
\caption{Bi-fidelity Algorithm}
\begin{algorithmic}[1]
\Require low-fidelity solver, high-fidelity solver, parameter space \(\mathcal{I}_z\), training set \(\Gamma= \{z_1, \dots, z_M\}\), desired number of nodes \(N\), new query parameter \(z\).
\State \textbf{Offline phase}
\State \textbf{Step 1:} Compute low-fidelity snapshots \(Q^L = [\mathbf{U}^L(z_1), \dots, \mathbf{U}^L(z_M)] \in \mathbb{R}^{2N_l \times M}\).
\State \textbf{Step 2:} Choose \(z_{i_1}\) with the largest variance in \(Q^L\), set \(\gamma_N = \{z_{i_1}, \dots, z_{i_N}\}\). 
\State \textbf{Step 3:} Build low-fidelity matrix \(\Phi^L = [\mathbf{U}^L(z_{i_1}), \dots, \mathbf{U}^L(z_{i_N})] \in \mathbb{R}^{2N_l \times N}\).
\State \textbf{Step 4:} Run high-fidelity solver on \(\gamma_N\): \(\Phi^H = [\mathbf{U}^H(z_{i_1}), \dots, \mathbf{U}^H(z_{i_N})] \in \mathbb{R}^{2N_h \times N}\).
\State \textbf{Step 5:} Precompute the Gram matrix \(\mathbf{G}^L = (\Phi^L)^T \Phi^L\).
\State \textbf{Online phase}
\State \textbf{Step 1:} Run low-fidelity solver obtain

\(\mathbf{U}^L(z) \approx \mathcal{P}_{\mathbf{U}^L(\gamma_N)}[\mathbf{U}^L(z)]=\sum_{k=1}^N c_k(z)\,\mathbf{U}^L(z_{k})=\Phil\,\mb{c}(z)\), where \( \mathcal{P}_{\mathbf{U}^L(\gamma_N)} \) denotes the projection onto the subspace \( \mathbf{U}^L(\gamma_N) \).
\State \textbf{Step 2:} Solve the coefficient vector \(\mathbf{c}(z)\) by solving the equation \(\mathbf{G}^L \mathbf{c}(z) = \mathbf{b}^L(z)\), where the projection vector is: \(\mathbf{b}^L(z) = (\Phi^L)^T \mathbf{U}^L(z)\).
\State \textbf{Step 3:} Reconstruct the high-fidelity estimate by linear combination of high-fidelity snapshots
\Return \(\mathbf{U}^B(z)\)
\end{algorithmic}
\end{algorithm}
\section{Theoretical analysis}
    In this section, we will establish a detailed error estimate for the CSL snapshots bi-fidelity algorithm~3.1. 
From \eqref{eq:BF_approx}, it can be seen that \(\mathbf{U}^B\) represents the approximate solution of the bi-fidelity algorithm. 
In order to obtain a uniform error estimate of \(\mathbf{U}^H-\mathbf{U}^B\), 
we might as well draw on the relevant research experience of predecessors \cite{Jinerror2021, Lin2023Error, LIN202547},
inserting the low-fidelity solution \(\mathbf{U}^L\) to split the total error:
\begin{align}\label{err_bf}
\mathbf{U}^H(z) - \mathbf{U}^B(z) 
&= \underbrace{[\mathbf{U}^H(z) - \mathbf{U}^L(z)]}_{=: \delta_{HL}(z)} 
+ \underbrace{\left[\mathbf{U}^L(z) - \sum_{k=1}^K c_k^L(z) \mathbf{U}^L(z_{i_k})\right]}_{=: e_{\text{proj}}(z)} \notag \\
&\quad + \underbrace{\sum_{k=1}^K c_k^L(z)\left[\mathbf{U}^L(z_{i_k}) - \mathbf{U}^H(z_{i_k})\right]}_{=: e_{\text{coef}}(z)}.
\end{align}
The decomposed errors consist of three components \(\delta_{HL}\), \(e_{\text{proj}}\), and \(e_{\text{coef}}\), which represent the error between the high-fidelity and low-fidelity solutions, projection error of the low-fidelity solution, and coefficient propagation error respectively. Next, we will analyze them one by one.

\subsection{Notations and prerequisites}
To commence, we establish the functional space and the associated norms employed throughout our study. 
The Hilbert space in question is characterized as follows:
\[
H\left(\mathbb{R}^{d_z}; \pi dz\right) = \left\{ f \mid I_z \to \mathbb{R}, \int_{I_z} f^2(z) \pi(z) dz < \infty \right\},
\]
which is additionally endowed with the inner product \(\langle f, g \rangle_\pi = \int_{I_z} f g \pi(z) dz\).
Then the standard \(L^2\)-norm is defined by
\[
\|f\|_{L^2(\Omega)} = \left( \int_{\Omega_x} \int_{\Omega_v} |f(x, v)|^2 \mathrm{d}v \mathrm{d}x \right)^{\frac{1}{2}}.
\]
For all \(z \in I_z\), the initial data \(f_0(\cdot, \cdot, z)\) belongs to \(C_{\text{per},c}^{\max\{p+1,3\}}\), and this norm is uniformly bounded with respect to \(z\). Moreover, \(f_0\) is nonnegative and compactly supported in velocity on \([-v_{max}, v_{max}]\).

Next, we present several important lemmas as reference points, in order to better analyze the error bounds of the bi-fidelity algorithm. More specifically, the errors arising from the macro/micro models are reviewed in Lemma 4.1, and the error bounds corresponding to the CSLDG method are reviewed in Lemma 4.2, under the assumption that the aforementioned conditions are satisfied; detailed derivations/proofs of these important lemmas are provided in the Supplementary Materials (SM).
To simplify the notation, the following lemmas are presented in the form of deterministic problems.

\begin{lemma}\cite{Brenier01012000}
Assume the initial distribution of the VP system \(f^\varepsilon(0, x, v)\) and the initial electric field satisfy the initial-modulated energy \(H_\varepsilon(0) \le C_0\),
for any \(t \in [0, T]\), the joint deviation between the macroscopic state \((\rho^\varepsilon, u^\varepsilon)\) and \((\rho_e, u_e)\) is uniformly controlled by the initial-modulated energy:
\[
\mathcal{D}^2_{macro} \left( (\rho^\varepsilon(t), u^\varepsilon(t)), (\rho_e(t), u_e(t)) \right) \le 2H_\varepsilon(0) \exp \left( 2 \int_{0}^t \left\| d(u_e)(\tau) \right\|_{L^\infty(\mathbb{R}^d)} d\tau \right)\le C',
\]
where the modulated total energy of the microscopic solution \((f^\varepsilon, E^\varepsilon)\) with respect to the macroscopic strong solution \((u_e, E_e)\) is shown in SM1.
\end{lemma}

\begin{lemma}\cite{Einkemmer2014}
The initial value \(f_0 \in C^{\max\{p+1,3\}}\) is nonnegative and compactly supported in velocity. 
For the numerical solution of \eqref{eq:VP_random} employing high-fidelity solver \(P^p\)-CSLDG, there exists a constant \(C > 0\), depending only on \(T\) and independent of \(\tau, h, n\), such that for all \(0 \le n\tau \le T\), the global error satisfies the following bound:
\begin{equation}
\sup_{0 \le n \le N} \left\| \left( \prod_{k=0}^{n-1} \widetilde{S}_k \right) \widetilde{f}_0(\cdot, \cdot) - f(n\tau, \cdot, \cdot) \right\|_{L^2_{x,v}} \le C \left( \tau^2 + h^{p+1} + \frac{h^{p+1}}{\tau} \right),
\end{equation}
where \(p\) represents the order of the polynomial, \(\tau=\Delta t\), \(h=\Delta x\), and \(N_t \tau = T\).
\end{lemma}
\begin{proof}
    A detailed proof of this lemma can be found in SM2. Their corresponding macroscopic quantities, such as \(\rho\), \(u\), also satisfy this error estimate. The reader can refer to a simple proof on their own.
\end{proof}

\subsection{Main Error Estimate}
Here, the CSLDG method is used as a high-fidelity solver, while the second-order MUSCL is chosen as a low-fidelity solver. We now present the main results of Theorem 4.4 on the error estimation of the bi-fidelity method under the above notation and prerequisites.
\LL{We denote $\mathbf{U}^{\text{VP}}$ the exact moments quantities derived from the VP equation  \(\mathbf{U}^{\text{EP}}\) the exact moment solutions of the EP equation.}
\begin{assumption}\cite{Cohen2015}
Let the random variable \( z := (z_j)_{j \geq 1} \) have components each supported on a compact set. Suppose that the random initial data \( f_{\mathrm{in}} \) admits an affine representer \( (f_j)_{j \geq 1} \), meaning that there exist a deterministic function \( \tilde{f}_0 = \tilde{f}_0(x) \) (independent of \( z \)) and a sequence of functions \( (g_j)_{j \geq 1} \) such that
\[
f_{\mathrm{in}}(z) = \tilde{f}_0 + \sum_{j \geq 1} z_j g_j,
\]
with the sequence \( \left\{ \|g_j\|_{L^\infty(\Omega)} \right\}_{j \geq 1} \) belonging to \( \ell^\beta \) for some \( 0 < \beta < 1 \), where \( \Omega \) denotes the physical domain, \LL{and the $x,v$ dependence is omitted in $f$ and $g$.}
\end{assumption}

\begin{theorem}
    Let \(\mathbf{U}^H\) be the high-fidelity CSLDG solution, and \(\mathbf{U}^B\) be the bi-fidelity approximation. Then, through \(K\) high-fidelity simulation runs, there exist constants \(C_1, C_2 > 0\) (which are independent of \(K, \tau, h\)) such that the error estimate of the bi-fidelity method is given by:
\begin{equation}
\left\| \mathbf{U}^H(z) - \mathbf{U}^B(z) \right\|_{L_x^2 L_z^2}
\le \frac{C_1}{(K/2 + 1)^{q/2}}
+ C_2 \sqrt{\frac{K}{\lambda_0}}\mathcal{Q},
\end{equation}
where \(\mathcal{Q}=\left( \tau_L+ h^2+\tau^2_H+h^{p+1} + \frac{h^{p+1}}{\tau_H} +C' \right),\) \(\tau_H\) and \(\tau_L\) represent the time steps of the high- and low-fidelity solvers, respectively.
\(q = 1/\beta - 1 > 0\), \(\lambda_0 > 0\) is the minimum eigenvalue of the Gram matrix.
\end{theorem}

\begin{proof}
From \eqref{err_bf}, applying the norm \(\| \cdot \|_{L_x^2 L_z^2}\) and triangle inequality, we get:
\begin{equation}\label{bi_error}
\|\mathbf{U}^H - \mathbf{U}^B\|_{L_x^2 L_z^2} 
\leq \| \delta_{HL} \|_{L_x^2 L_z^2} 
+ \| e_{\text{proj}} \|_{L_x^2 L_z^2} 
+ \| e_{\text{coef}} \|_{L_x^2 L_z^2}.
\end{equation}

The first term \(\| \delta_{HL} \|_{L_x^2 L_z^2}\) represents the error between the high-fidelity and low-fidelity solutions. We divide the error into high-fidelity scheme's error, model error, and low-fidelity scheme's error.
Now we directly use the conclusions of Lemma 4.1 and Lemma 4.2 and get
\begin{equation}\label{bi_error_1}
\begin{split}
 \|\mathbf{U}^H - \mathbf{U}^L\|_{L_x^2 L_z^2} & \le \|\mathbf{U}^H - \mathbf{U}^{\text{VP}}\|_{L_x^2 L_z^2} + \|\mathbf{U}^{\text{VP}} - \mathbf{U}^{\text{EP}}\|_{L_x^2 L_z^2} + \|\mathbf{U}^{\text{EP}} - \mathbf{U}^L\|_{L_x^2 L_z^2}\\
 & \le C \left( \tau^2_H + h^{p+1} + \frac{h^{p+1}}{\tau_H} \right)+C'+C(\tau_L+h^2)\\
 & \le C\left( \tau_L+ h^2+\tau^2_H+h^{p+1} + \frac{h^{p+1}}{\tau_H}\right) +C' .
 \end{split}
\end{equation}
The second term \(\|e_{\text{proj}} \|_{L_x^2 L_z^2}\) represents the projection error of the low-fidelity solution.
Let \(d_K(\mathbf{U}^L(I_z))\) denote the Kolmogorov \(K\)-width of the manifold \(\mathbf{U}^L(I_z)\), i.e.,
\(
d_K(\mathbf{U}^L(I_z)) = \inf_{\dim(V_K)=K}\; \sup_{u\in \mathbf{U}^L(I_z)} d^L(u, V_K).
\)
Take \(\mathcal{H}=H_x^s\), from previous studies \cite{Jinerror2021,Lin2023Error}, the projection error satisfies
\begin{equation*}
d_K(\mathbf{U}^L(I_z))_{\mathcal{H}}
\le \| \mathbf{U}^L - \sum_{\nu\in \Lambda_K} w_\nu \mathcal{P}_\nu \|_{L^\infty(I_z,\mathcal{H})}
\le \frac{C}{(K+1)^q}, \qquad q = \frac{1}{\beta} - 1,
\end{equation*}
where \(\beta\) is a constant associated with the affine representer \((f_j)_{j\ge1}\) in the random initial data (see Assumption 4.3), and the \(\sum_{\nu\in\Lambda_K} w_\nu P_\nu\) denotes the truncated Legendre expansion. 
Plug \(K/2\) into the above equation, we obtain
\begin{equation}\label{bi_error_2}
    \| e_{\text{proj}} \|_{L_x^2 L_z^2} \le \frac{C_1}{(K/2+1)^{q/2}}.
\end{equation}

\LL{In order to derive a bound for \( \| \mathbf{c} \| \) thereby estimate the third term  \(\| e_{\text{coef}}\|_{L_x^2 L_z^2}\), 
by integrating over the physical domain \(\Omega\), one obtains}
\begin{align*}
\int_{\Omega} \left( \mathcal{P}_{\mathbf{U}_j^L(\gamma_N)} \big[ \mathbf{U}_j^L(z) \big] \right)^2 dx
&= \sum_{k,s=1}^K c_k(z) c_s(z) \int_{\Omega} \mathbf{U}_j^L(z_k)\, \mathbf{U}_j^L(z_s)\, dx \\
&:= \mathbf{c}^T \mathbf{G}^L \mathbf{c}\ge \lambda_{j,0}\, \| \mathbf{c} \|^2,
\end{align*}
where \(\mathbf{G}^L\) denotes the Gramian matrix of \(\mathbf{U}_j^L(\gamma_N)\) as defined in bi-fidelity Algorithm (with the inner product taken in \(L_x^2\)), and \(\lambda_{0} > 0\) be the smallest eigenvalue of \(\mathbf{G}_j^L\).
\LL{For each \(z\) it can be derived that 
\begin{equation}\label{term3_c}
    \left( \sum_{k=1}^K \| c_k(z) \|_{L^2_z}^2 \right)^{1/2}
\lesssim \frac{1}{\sqrt{\lambda_{0}}} \int_{\Omega} \left( \mathcal{P}_{\mathbf{U}_j^L(\gamma_N)} \big[ \mathbf{U}_j^L(z) \big] \right)^2 dx\lesssim \frac{1}{\sqrt{\lambda_{0}}}\, \big\| \mathbf{U}_j^L(z) \big\|_{L^2_xL^2_z}.
\end{equation}
Please refer the details in \cite{Jinerror2021,LIN202547}.}



Now, we using \eqref{bi_error_1} and \eqref{term3_c}, with the Cauchy-Schwarz inequality, one has
\begin{equation}\label{bi_error_3}
    \begin{split}
        \| e_{\text{coef}} \|_{L_x^2 L_z^2} &\leq \left( \sum_{k=1}^K \|c_k(z)\|^2 \right)^{1/2} \left( \sum_{k=1}^K \|\mathbf{U}^H(z_k) - \mathbf{U}^L(z_k)\|_{L^2_x L^2_z}^2 \right)^{1/2} \\
&\leq \frac{\sqrt{K}}{\sqrt{\lambda_0}} \|\mathbf{U}^L(z)\|_{L^2_x L^2_z}\left( \max_k \|\mathbf{U}^H(z_k) - \mathbf{U}^L(z_k)\|_{L^2_x}^2 \right) \\
&\leq C \sqrt{\frac{K}{\lambda_0}} \left( \tau_L+ h^2+\tau^2_H+h^{p+1} + \frac{h^{p+1}}{\tau_H} +C' \right).
    \end{split}
\end{equation}

\LL{Finally, let \(\mathcal{Q}= \left( \tau_L+ h^2+\tau^2_H+h^{p+1} + \frac{h^{p+1}}{\tau_H} +C' \right) \) and plug \eqref{bi_error_1}, \eqref{bi_error_2} and \eqref{bi_error_3} into \eqref{bi_error}, we derive that}
\begin{equation*}
    \|\mathbf{U}^H - \mathbf{U}^B\|_{L_x^2 L_z^2} 
\leq C\mathcal{Q}
+ \frac{C_1}{(K/2+1)^{q/2}}
+C \sqrt{\frac{K}{\lambda_0}}\mathcal{Q}
\leq \frac{C_1}{(K/2+1)^{q/2}}
+C_2 \sqrt{\frac{K}{\lambda_0}}\mathcal{Q}.
\end{equation*}
\end{proof}
\section{Numerical experiments}
The numerical experiments in this section systematically evaluate the performance, applicability, and scalability of the proposed structure-preserving SL bi-fidelity framework. 
For deterministic problems, we use CSLDG to solve the VP system, examining spatial convergence order, conservation properties, and the resolution of fine phase-space structures, such as vortex filaments and stripes. 
For uncertain problems, we introduce random parameters into the initial conditions, thereby extending the model to a VP model with stochastic variables, focusing on verifying the robustness and statistical convergence of the bi-fidelity algorithm. 
In addition, we extend the algorithm to more challenging two-component (electron-ion) plasma systems and higher-dimensional phase-space settings, and the numerical results demonstrate good stability and scalability of the algorithm. 
For all simulations, unless otherwise specified, 
the computational domain is $[0, L] \times [-v_{max}, v_{max}]$, and the spatial domain length is set as \(L = \frac{2\pi}{\kappa}\).

\subsection{Deterministic problem}
We considere by CSLDG for application in solving a series of classical deterministic problems of the VP model, and the scheme was extended to high-dimensional tests, including Landau damping and two-stream instability. 
Let the time step size  \(\Delta t= \frac{\text{CFL}}{\frac{v_{max}}{\Delta x} + \frac{max(|E|)}{\Delta v}}\).

For problems without exact solutions, we estimate the convergence rate by refine the mesh step by step. 
Specifically, let \(f_{h}\) represent the numerical solution of the mesh split under \(N_h\). The error approximates the difference between the solutions of the continuous refinement hierarchy. 
We employ both the discrete \(L^2\)-norm \(\|\cdot\|_{L^2}\) and the discrete \(L^\infty\)-norm \(\|\cdot\|_{L^\infty}\). The experimental orders of convergence are then defined as
\[
\mathrm{Order} = \frac{\log \|f_{h} - f_{h/2}\|_{L^\eta} - \log \|f_{h/2} - f_{h/4}\|_{L^\eta}}{\log 2}, \quad \eta = 2, \infty.
\]

\subsubsection{1D1V}
\begin{example}[Landau damping]
\begin{align}\label{1D1V:ex1}
f_0(x, v) = \frac{1}{\sqrt{2\pi}} (1 + \alpha \cos(\kappa x)) \exp\left(-\frac{v^2}{2}\right),
\end{align}
where \(\alpha=0.01\) and \(0.5\) correspond to the weak and strong Landau damping, respectively.
We truncate the velocity domain at $v_{max} = 2\pi$, and \(\kappa = 0.5\).
\end{example}

Table~\ref{tab:convergence} reports the convergence study, which validates the spatial accuracy of the proposed CSLDG solver. 
We set $\mathrm{CFL}=0.01$ and choose the time step as $\Delta t=\mathrm{CFL}\cdot\left(\min\{\Delta x,\Delta v\}\right)^{p+1}$, so that the temporal error is negligible and the spatial error dominates. 
For the $P^0$-, $P^1$-, and $P^2$-CSLDG schemes, the error norms decrease at rates of approximately first, second, and third order, respectively, as the mesh is refined. 
These observed rates match the theoretical order $p+1$ and are consistent with the error analysis in Section~4.1.

\begin{table}[htbp]
\centering
\caption{Weak Landau damping with \( T = 1 \). Order of accuracy in space for CSLDG solver (CFL=0.01).}
\label{tab:convergence}
\begin{tabular}{c c c c c c}
\toprule
& Mesh & \multicolumn{1}{c}{\(L^2\) error} &  \multicolumn{1}{c}{Order} & \multicolumn{1}{c}{\(L^\infty\) error} & \multicolumn{1}{c}{Order} \\
\midrule
& 16   & 5.2589E-01 & --   & 8.2046E-02 & --      \\
\(P^0\)-CSLDG & 32   & 2.5579E-01 & 1.04 & 4.6228E-02 & 0.83    \\
& 64   & 1.2991E-01 & 0.98 & 2.4191E-02 & 0.93    \\
& 128  & 6.5238E-02 & 0.99 & 1.2087E-02 & 1.00    \\
\midrule
& 16   & 1.1031E-01 & --   & 2.8982E-02 & --      \\
\(P^1\)-CSLDG & 32   & 3.5694E-02 & 1.63 & 1.0655E-02 & 1.44    \\
& 64   & 9.1707E-03 & 1.96 & 3.4030E-03 & 1.65    \\
& 128  & 2.3843E-03 & 1.94 & 1.0211E-03 & 1.74    \\
\midrule
& 16   & 5.7414E-02 & --   & 1.7433E-02 & --      \\
\(P^2\)-CSLDG & 32   & 3.8824E-03 & 3.89 & 1.0924E-03 & 4.00    \\
& 64   & 5.0956E-04 & 2.93 & 1.5403E-04 & 2.83    \\
& 128  & 6.8032E-05 & 2.90 & 2.1819E-05 & 2.82    \\
\bottomrule
\end{tabular}
\end{table}

Fixed mesh of $160\times 160$ elements and set $\text{CFL}=10$, the $P^1$-CSLDG scheme is applied to simulate the Landau damping at \(T=40\).
Fig.\ref{dp_landau damping_f} show the phase‑space distribution remains smooth and Gaussian‑like under weak damping, while under strong damping the scheme clearly resolves the emerging filamentary structures. 
Fig.\ref{dp_landau damping_error} shows the time evolution of the electric field energy. 
For different CFLs, the weak damping case (\(\alpha=0.01\)) exhibits a decay rate matching the theoretical value \(\gamma=-0.153\), similarly, in the strong damping case (\(\alpha=0.5\)), regardless of the CFL number, the electric energy exhibits a decay rate of \(\gamma_1=-0.287\) followed by a subsequent growth rate of \(\gamma_2=0.085\), in agreement with~\cite{Cai2018,FILBET2001166,Heath2012}.

From the perspective of the conservation proper Fig.\ref{dp_landau damping_error}, the numerical scheme strictly preserves mass conservation, with mass errors close to machine precision. 
However, momentum conservation cannot be strictly guaranteed and only exhibits relatively good numerical conservation. 
Especially under strong Landau damping, momentum errors show severe oscillations and increase. 
Combined with the analysis in Remark 2, this phenomenon may be due to smaller CFL numbers leading to more time steps, which in turn increases error accumulation.

\begin{figure}[!htbp]
    \centering
     \includegraphics[width=0.48\linewidth]{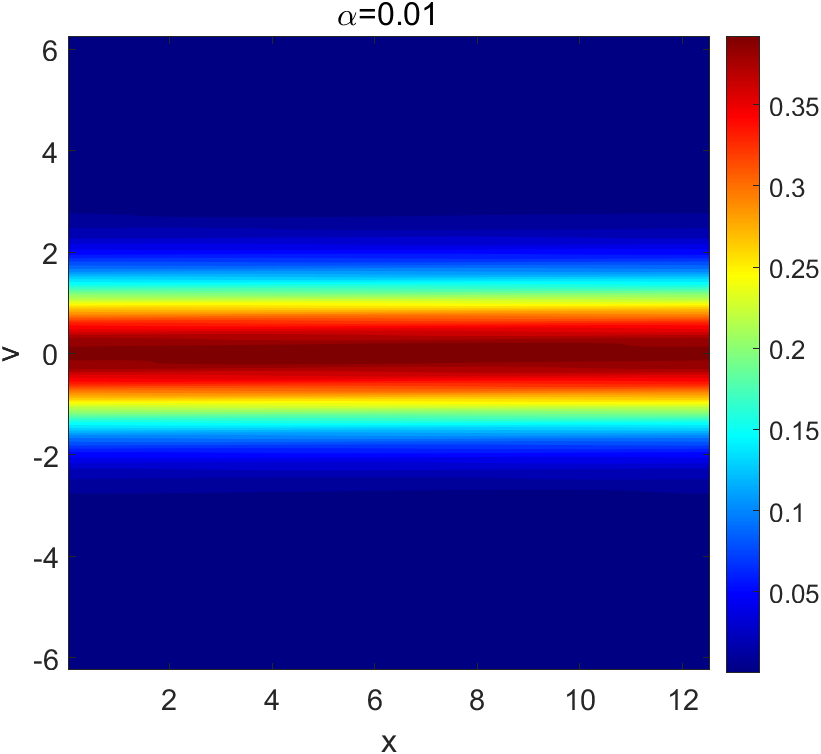}
     \includegraphics[width=0.48\linewidth]{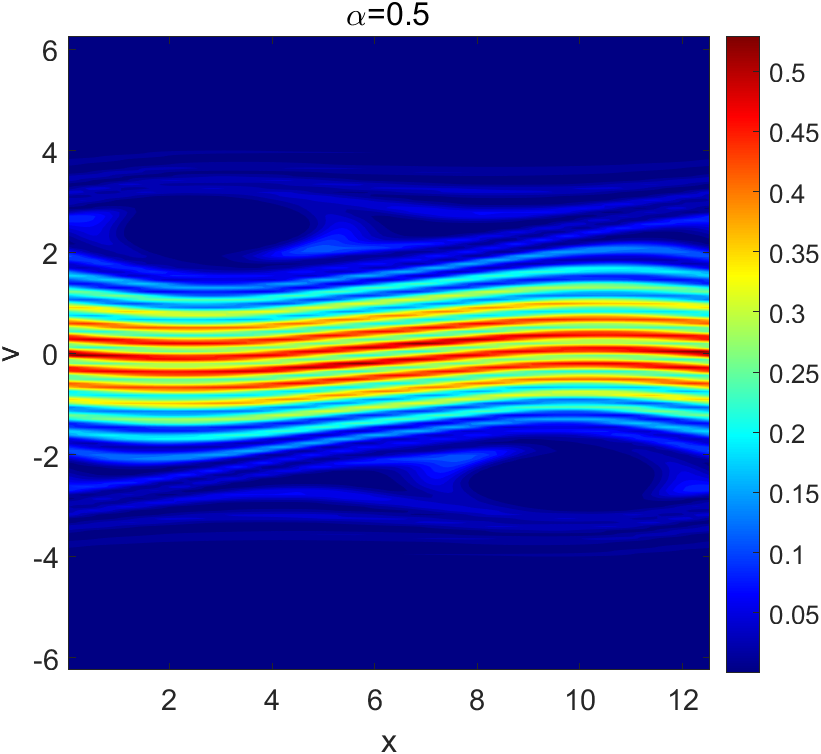}
    \caption{Weak(left)/Strong(right) landau damping: The distribution function \(f\) simulated by $P^1$-CSLDG at $T = 40$.} 
    \label{dp_landau damping_f}
\end{figure}

\begin{figure}[!htbp]
    \centering
    \includegraphics[width=0.48\linewidth]{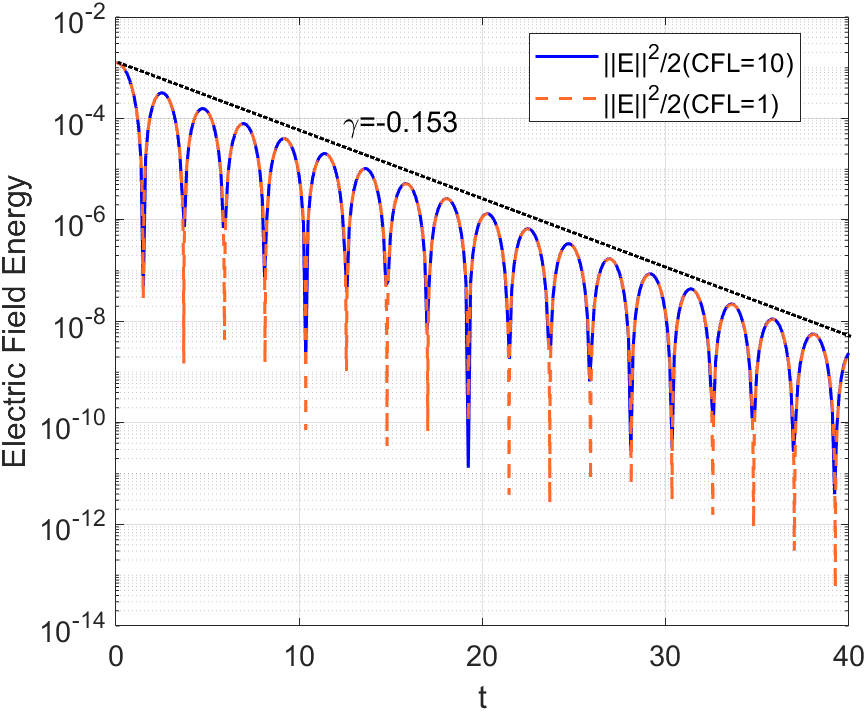}
    \includegraphics[width=0.48\linewidth]{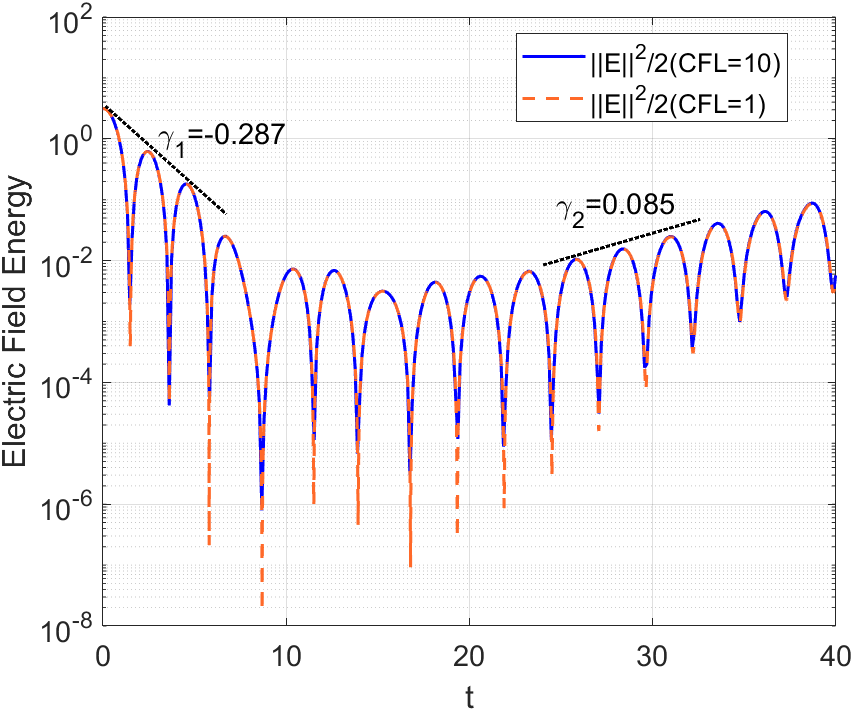}\\
    \includegraphics[width=0.48\linewidth]{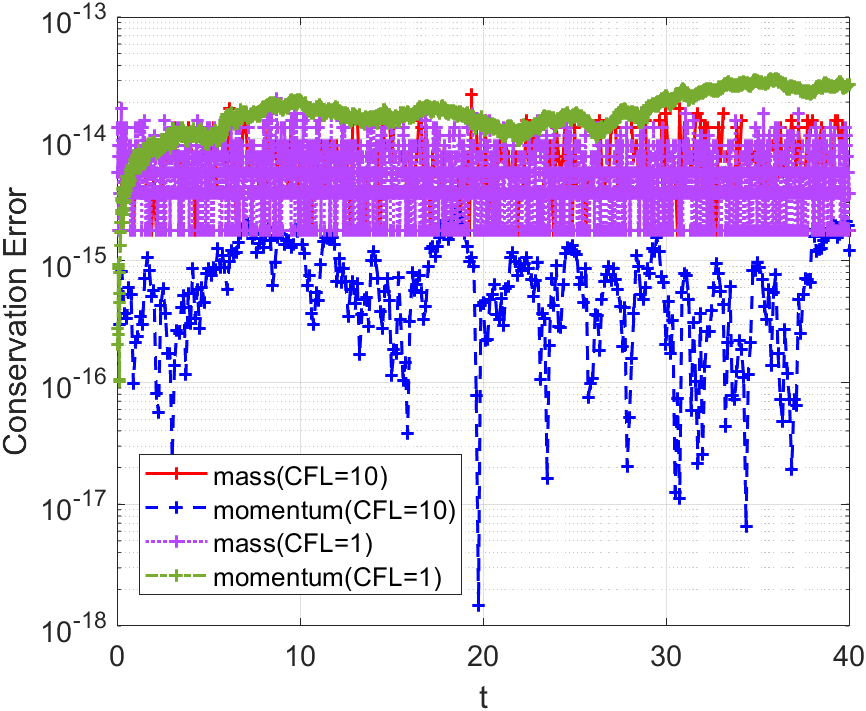}
    \includegraphics[width=0.48\linewidth]{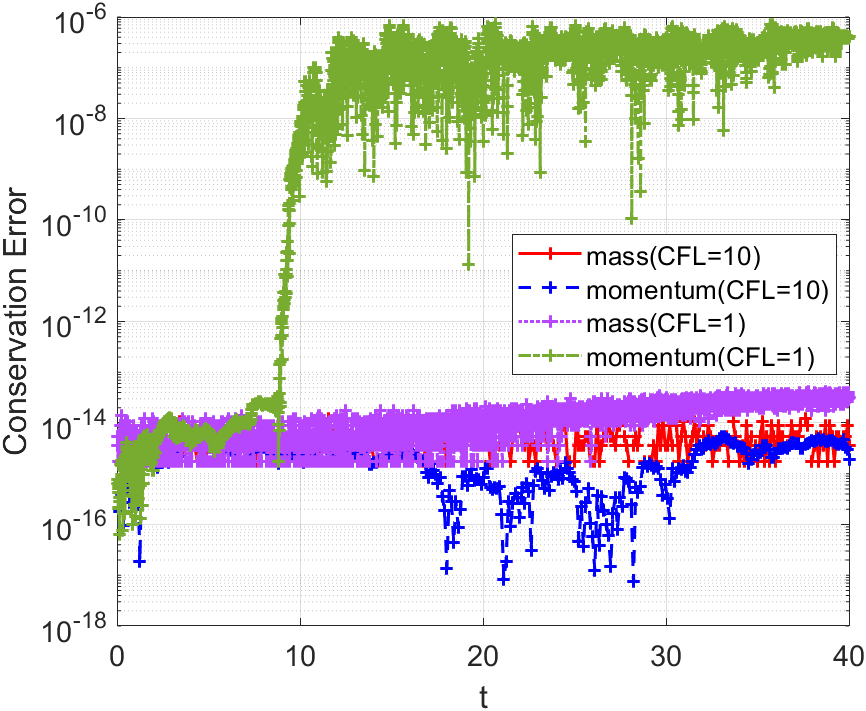}
    \caption{Weak(left)/Strong(right) landau damping: Time evolution of the electric energy (Top); Time evolution of the error in conserved quantities (bottom).}
    \label{dp_landau damping_error}
\end{figure}

\subsubsection{2D1V}
In this part, we proceed to apply the CSLDG slover to a 2D1V (two space dimensions and one velocity dimension) electrostatic plasma wave simulation. The objective is to assess the numerical accuracy of the scheme and to examine its conservation behavior when employed in a multi-dimensional setting.

Now, we give a brief presentation of the scheme's application to 2D1V simulations. 
The Strang operator splitting is adopted, with each time step decomposed into six substeps, like the algorithm in Section~3.1 illustrates the procedure.
The procedure proceeds as follows: 
first, an \(x\)-sweep of half a time step (\(\Delta t/2\)) is performed, transporting the distribution in the \(x\)-direction with speed \(v\). 
Second, a \(y\)-sweep of \(\Delta t/2\) transports the distribution in the \(y\)-direction with the same speed \(v\). 
Third, the two-dimensional Poisson equation is solved to compute the electric field \(E\) from the updated charge density \(\rho\). 
Fourth, a full-step \(v\)-sweep (\(\Delta t\)) accelerates the particles in velocity space using the computed \(E\). 
Fifth, another \(y\)-sweep of \(\Delta t/2\) is carried out. 
Finally, the sixth substep performs another \(x\)-sweep of \(\Delta t/2\), completing the time step. 
This symmetric sequence ensures second-order temporal accuracy characteristic of Strang splitting.

\begin{example}[Multi‑vortex two‑stream instability \cite{Liu2023}]
\begin{align}\label{ex2}
f_0(x,y,v) = \frac{1}{2v_{\text{th}}\sqrt{2\pi}}
\Bigl[\exp\!\bigl(-\tfrac{(v-u)^2}{2v_{\text{th}}^2}\bigr) +
\exp\!\bigl(-\tfrac{(v+u)^2}{2v_{\text{th}}^2}\bigr)\Bigr]
\Bigl[1 + \alpha\bigl(\cos(\kappa_x x)+\cos(\kappa_y y)\bigr)\Bigr],
\end{align}
with parameters \(u = 0.99\), \(v_{\text{th}} = 0.3\), \(\alpha = 0.05\), and \(\kappa_x = \kappa_y = 2/13\), the velocity domain at $v_{max} = 2\pi$.
\end{example}

Fixed mesh of $128\times128\times128$ elements and set $\text{CFL}=10$, the phase-space structures of the multi-vortex two-stream instability solution at the final simulation time \(T=40\) are visualized in Fig.\ref{dp_multi‑vortex two‑stream}. 
As observed from the cross-sectional slices and the 3D perspective, the initially symmetric double-hump distribution undergoes severe nonlinear deformation driven by the two-stream instability. 
The instability triggers the formation of characteristic phase-space holes (vortices) and intense filamentation, indicating strong particle-trapping effects.

The corresponding evolution of the electrostatic energy and the conservation properties is illustrated in Fig.\ref{dp_multi‑vortex two‑stream_conservation}. As shown in the left panel, the electric field energy initially undergoes exponential growth driven by the linear unstable modes, before entering a nonlinear saturation phase. 
After saturation, the energy levels exhibit complex, multi-frequency oscillations that persist over the long-time evolution.
The right panel highlights the global numerical errors for both mass and momentum are maintained close to machine precision, validates the high accuracy, long-term stability, and conservative nature of the numerical method in resolving complex VP dynamics.
\begin{figure}[!htbp]
    \centering
    \includegraphics[width=0.49\linewidth]{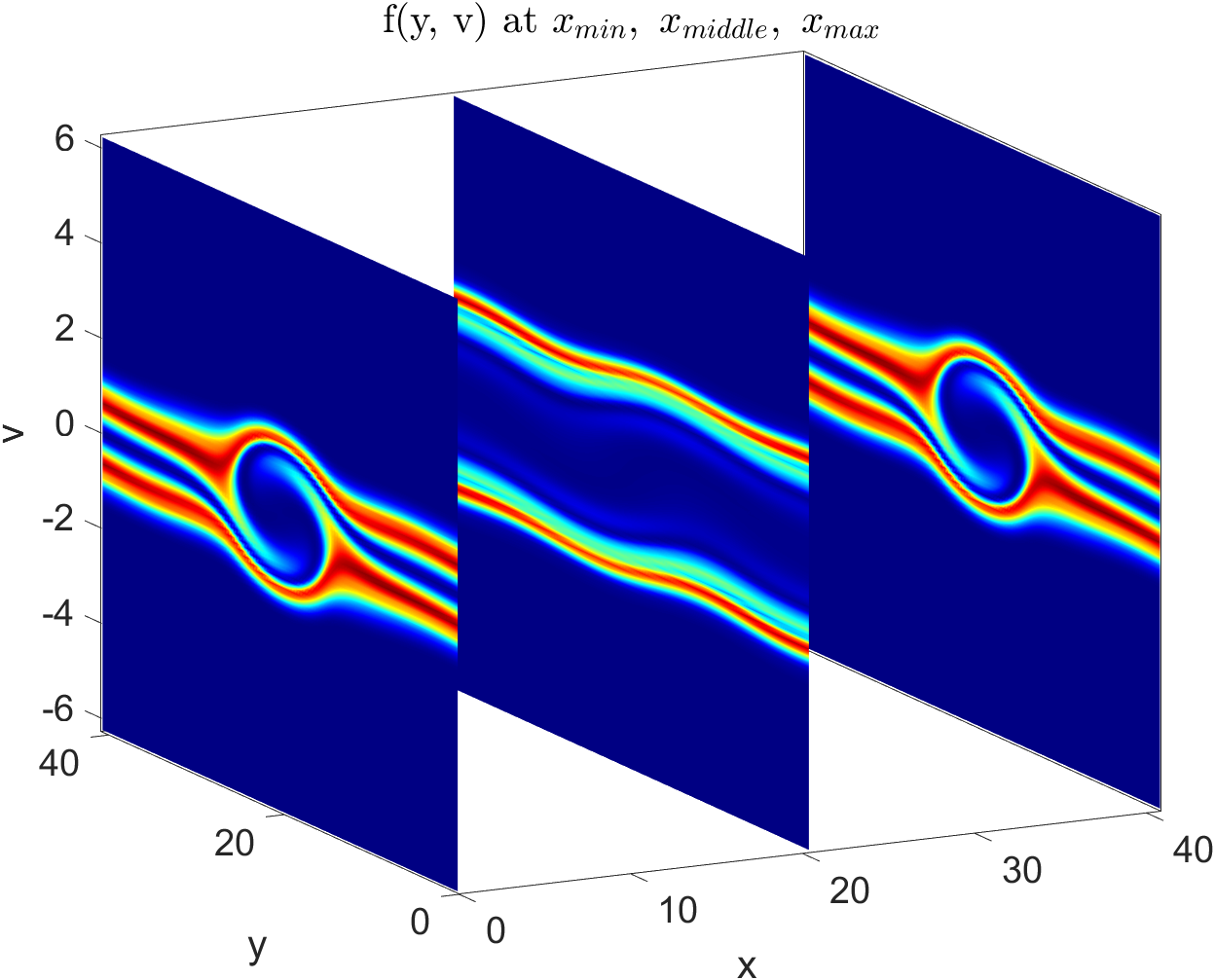}
    \includegraphics[width=0.49\linewidth]{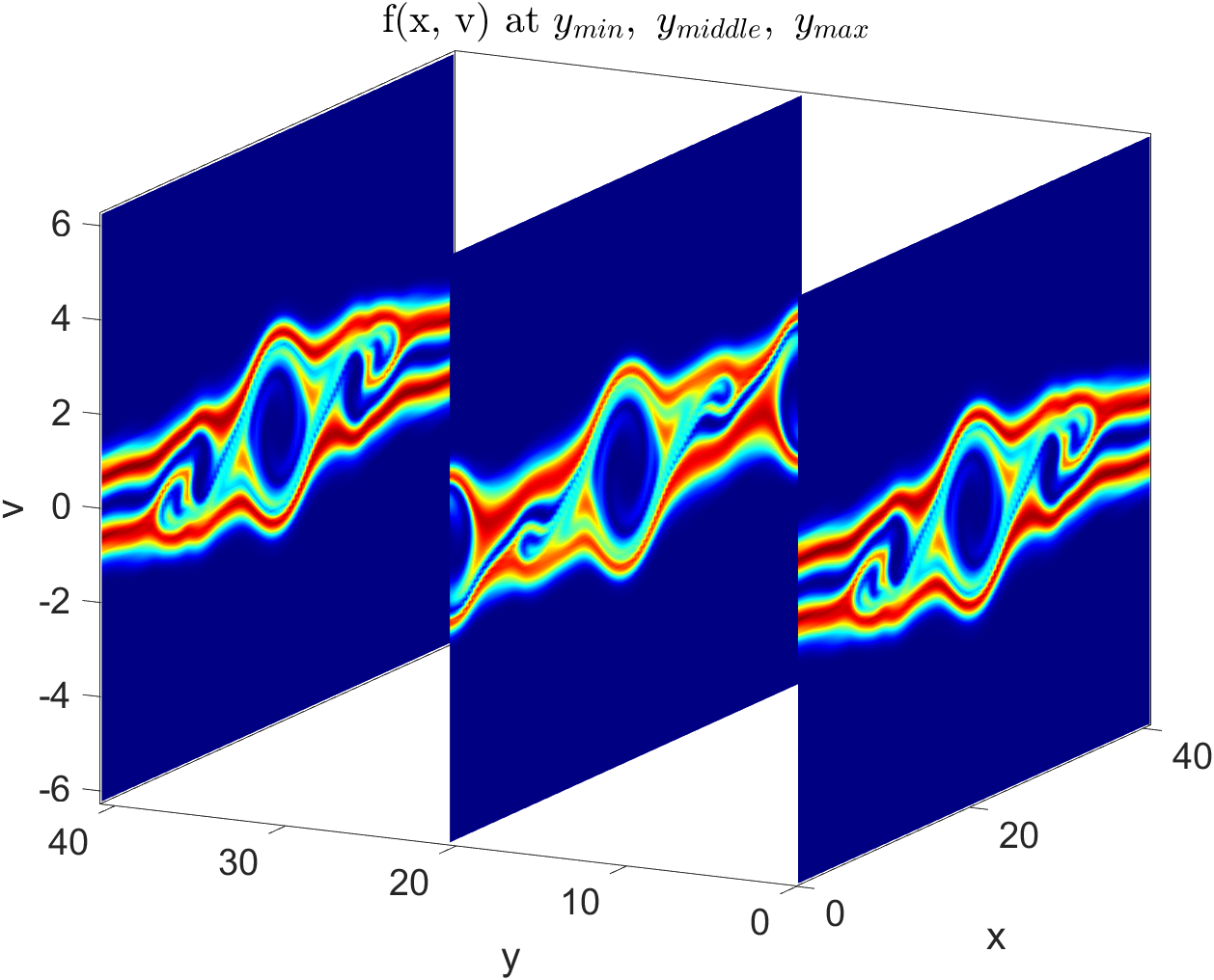}\\
    \includegraphics[width=0.5\linewidth]{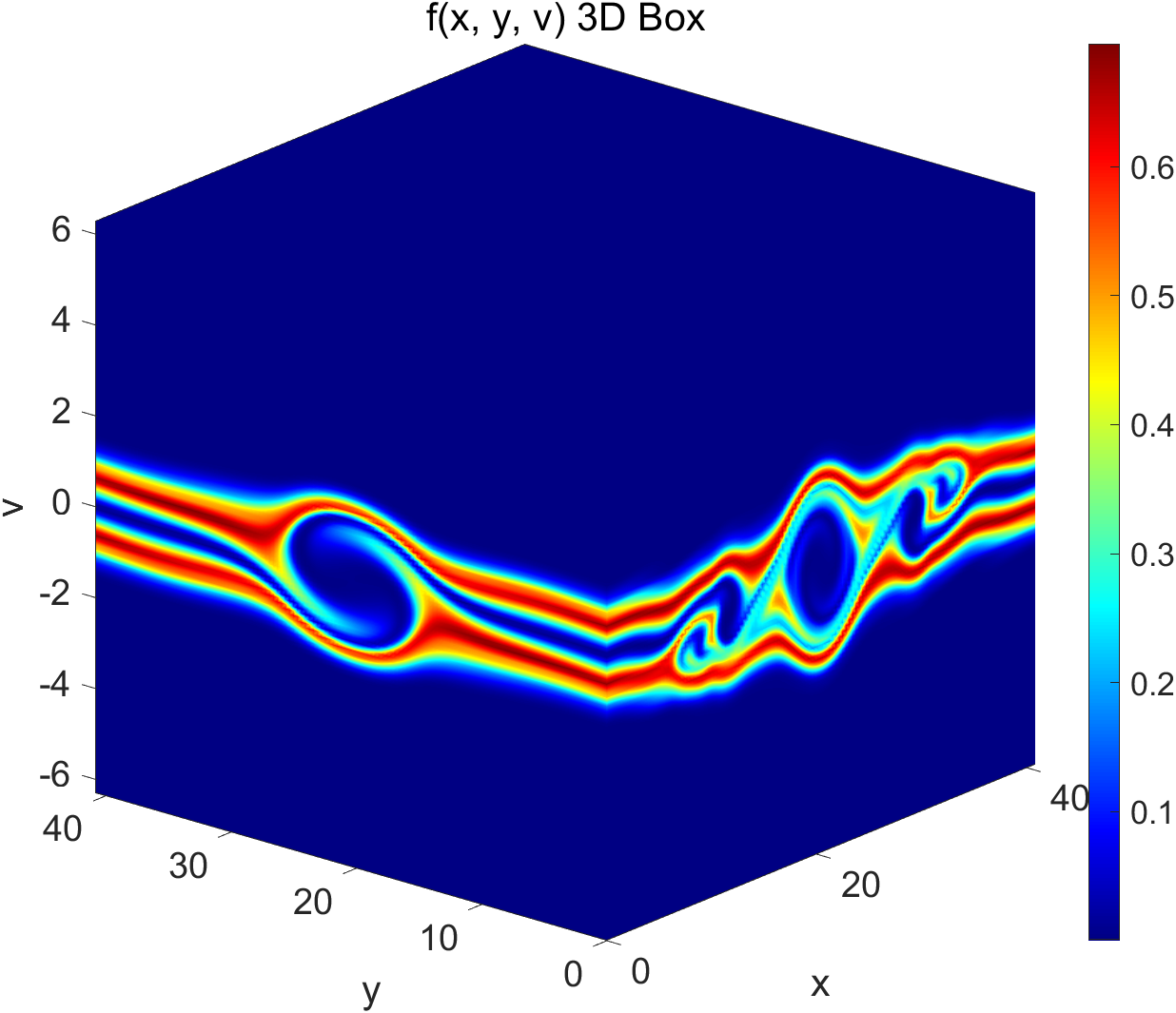}
    \caption{Snapshots of the Multi-vortex two-stream instability distribution function \(f(x, y, v)\) simulated by $P^1$-CSLDG at \(T=40\): The top row displays cross-sections at different \(x, y\)-coordinates; the bottom panel shows the 3D box view of the phase-space vortices.}
    \label{dp_multi‑vortex two‑stream}
\end{figure}

\begin{figure}[!htbp]
    \centering
    \includegraphics[width=0.49\linewidth]{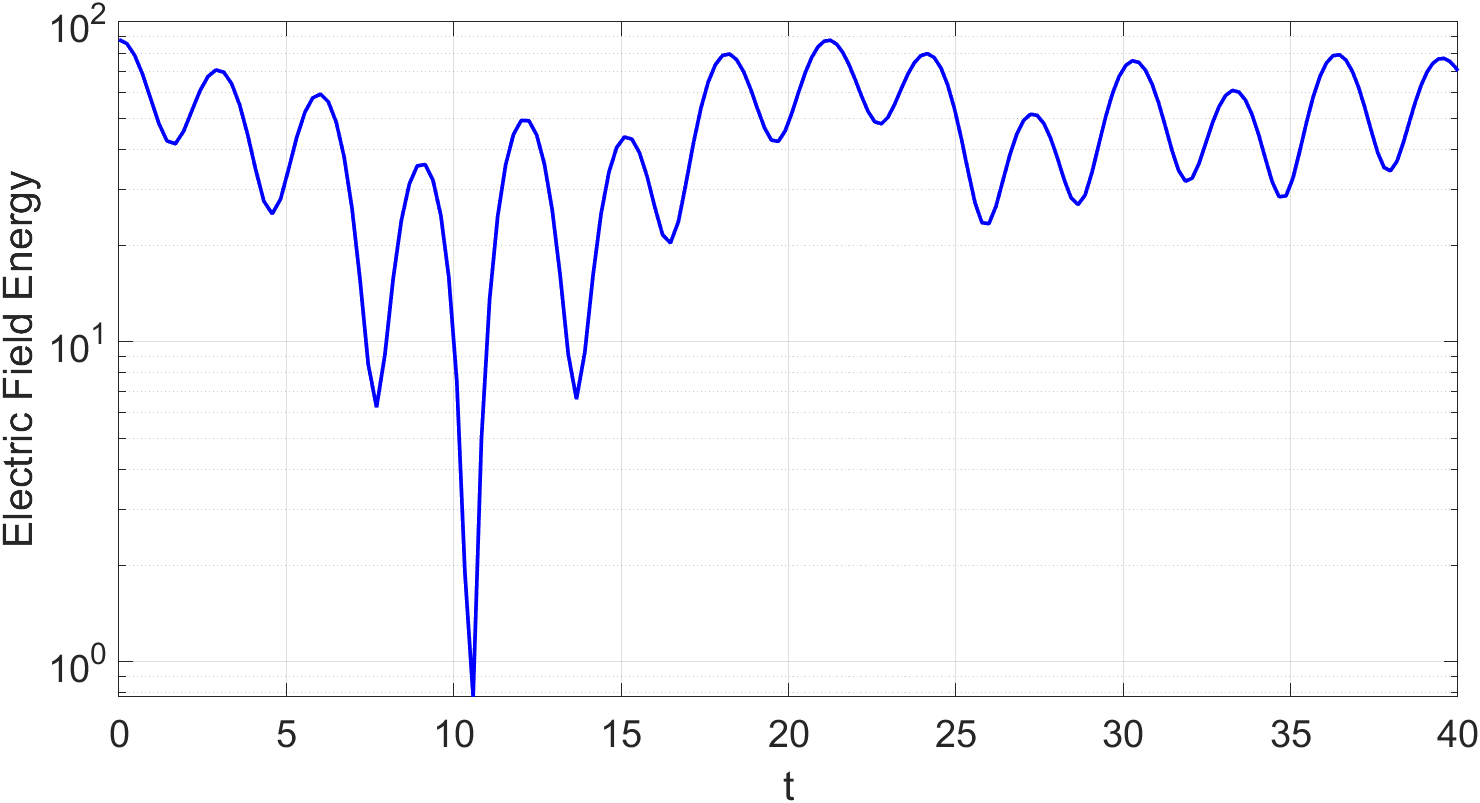}
    \includegraphics[width=0.49\linewidth]{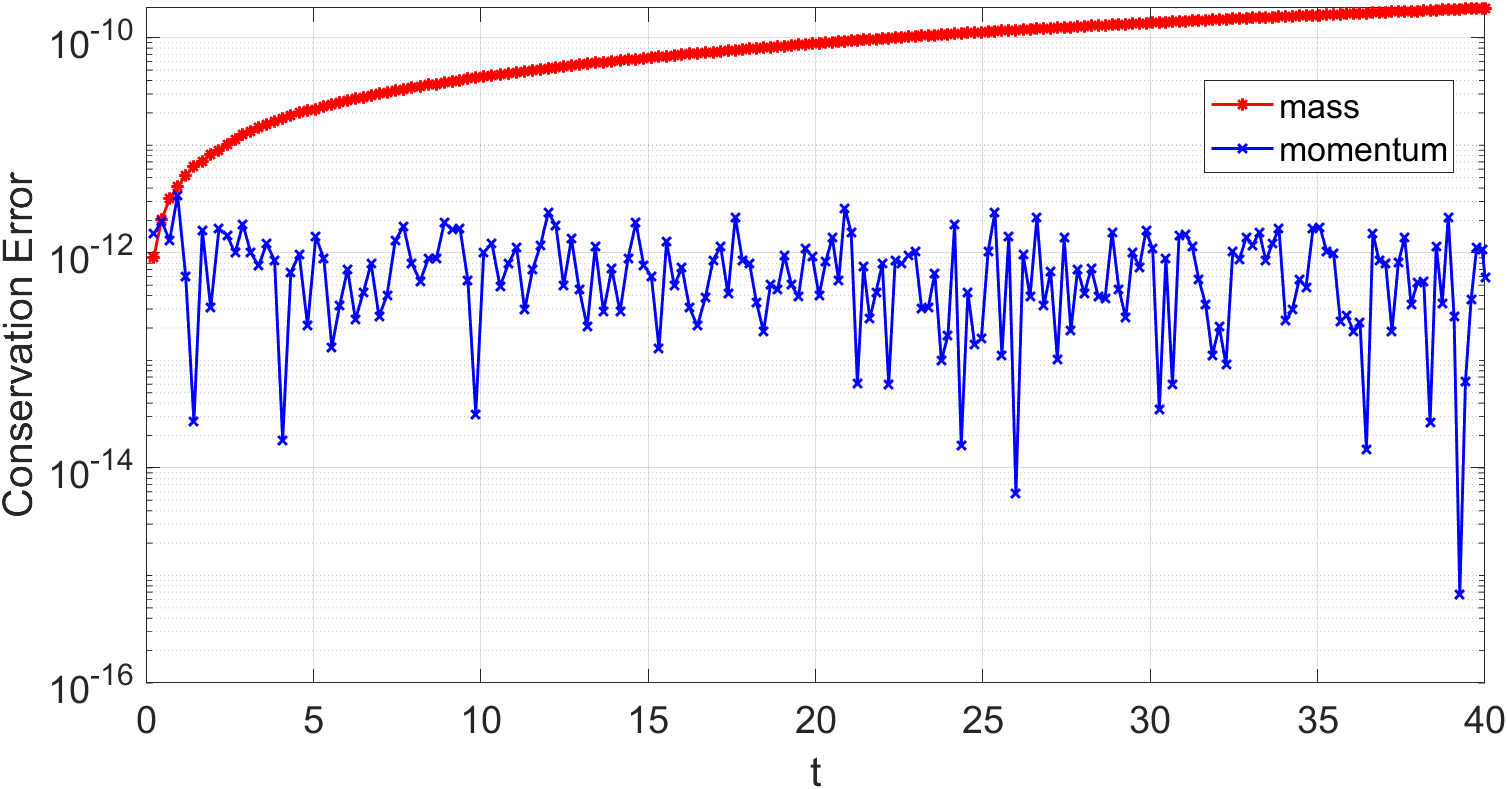}
    \caption{Multi-vortex two-stream instability: Time evolution of the electric energy (left) and the error in conserved quantities (right).}
    \label{dp_multi‑vortex two‑stream_conservation}
\end{figure}

\subsection{Uncertain problem}
In this section, we extend study to VP system with random parameters, exploring the efficiency and accuracy of the bi-fidelity method we introduced.
The high-fidelity model is solved using the $P^1$-CSLDG method, at this point, for the bi-fidelity framework, the $P^1$-CSLDG can still provide a sufficiently accurate reference.
The low-fidelity model is solved using second-order MUSCL and the time step set \(\Delta t = 0.01\). 

For the complete initial conditions with uncertain parameters---both the microscopic
distribution functions $f_0$, the corresponding macroscopic moments $\rho_0$ and $ (\rho u)_0$.
To assess the accuracy of our method,
given $K$ dependent test samples $\{z_j^\text{test}\}_{j=1}^{K}$,
the average $L^2$ error for each physical observable is estimated as
\begin{equation}
  \label{eq:err_estimator}
  \mathrm{err}_{(\rho, u)}
  =\left\| \mathbf{U}^H(z) - \mathbf{U}^B(z) \right\|_{L_x^2 L_z^2}
  = \frac{1}{K}\sum_{j=1}^{K}
  \sqrt{\frac{1}{N_h}\sum_{l=1}^{N_h}
  \bigl([\mathbf{U}^H(z_j)]_{l} - [\mathbf{U}^B(z_j)]_{l}\bigr)^2},
\end{equation}
the low-fidelity baseline error is computed using the same test
samples and error metric, with the high-fidelity solution serving
as the reference. Specifically, it is obtained by replacing
$\mathbf{U}^B$ with $\mathbf{U}^L$ in the above formula.
In all the examples below, periodic boundary conditions are assumed. 
Without loss of generality, we assume that the $d_z$ dimensional random variable $z$ follows a uniform distribution over $[-1, 1]^{d_z}$. 
The training set consists of \(M = 1000\) random samples of \(z\), using \(K = 100\) to check the error of the bi-fidelity approximation.

\subsubsection{Single-component}
In this example, consider introducing random parameters on the Maxwellian equilibrium state, and the corresponding macroscopic initial values can be obtained by taking the moments.
\begin{example}
\begin{align}
\begin{cases}
f_0(x,v,z) &= \dfrac{n_0}{\sqrt{2\pi T_0}} e^{-(v-u_0)^2/2} \Bigl[1 + S(x,z)\Bigr],\\
\rho_0(x,z) &= n_0\Bigl[1 + S(x,z)\Bigr],\\
(\rho u)_0(x,z) &= n_0u_0\Bigl[1 + S(x,z)\Bigr],
\end{cases}
\end{align}
    take strong Landau damping in the above deterministic problem as an example, \(S(x,z)=\sum_{m=1}^{2} \dfrac{\alpha_0}{m}\, z_m \cos(m\kappa x),\) let \(\alpha_0 = 0.5, \kappa = 0.5, \, n_0=1+0.05z_3, \, u_0=0.02z_4, T_0=1+0.05z_5,\) and \(d_z=5\), \(N_h=N_l=128\), the velocity domain $v_{max} = 2\pi$.
\end{example}

From the left sidebar of Fig.\ref{up_single-component}, we can clearly see the fast convergence between high-fidelity and bi-fidelity.
As the number of high-fidelity runs \(r\) increases, the $L^2$ errors of \(\rho\) and \(u\) drop significantly. With only about 10 high-fidelity runs, the bi-fidelity approximation can reach an accuracy of around \(\mathcal{O}(10^{-2})\).
In contrast, the low-fidelity error (shown as a red dashed line) with an error level of around \(\mathcal{O}(10^{-1})\).
To further illustrate the performance of our bi-fidelity method, we compared the low-fidelity, high-fidelity, and the corresponding bi-fidelity solutions (with \(r = 30\)) for a particular sample point \(z\) in the right column.
One observes that the high- and bi-fidelity solutions match well throughout the entire spatial domain, whereas the low-fidelity solutions appear to be quite inaccurate.
Fig.\ref{up_single-component_E_SD} clearly illustrates the mean and standard deviation of the bi-fidelity approximation for both \(\rho\) and \(u\), the bi-fidelity curves (red dots) precisely overlap the high-fidelity results (black lines). 
This result demonstrates that the bi-fidelity model can accurately capture both the deterministic structures and the stochastic characteristics across the random space.

\begin{figure}[!htbp]
    \centering
    \includegraphics[width=0.49\linewidth]{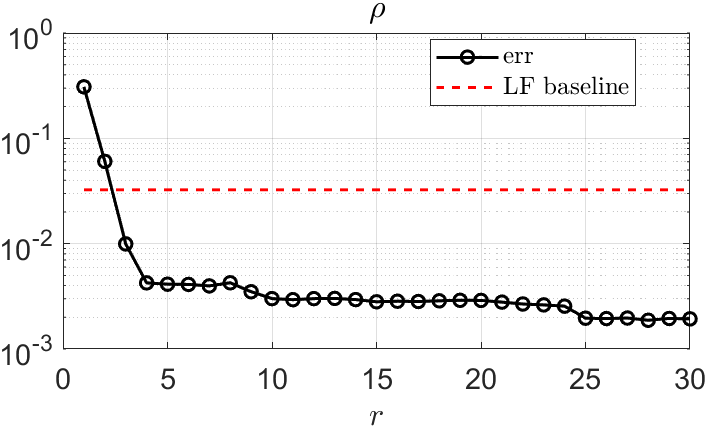}
    \includegraphics[width=0.48\linewidth]{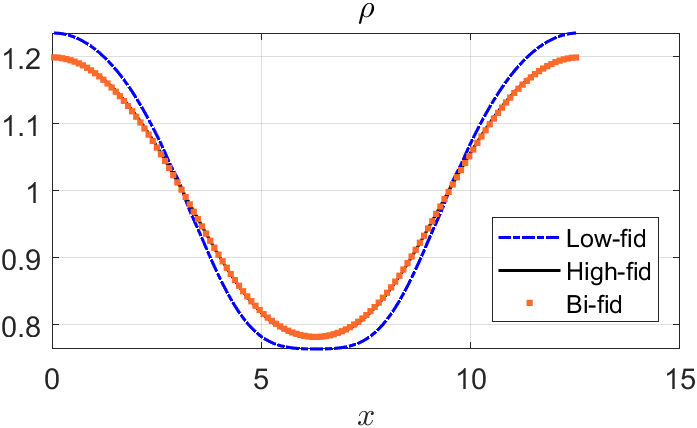}\\
    \includegraphics[width=0.48\linewidth]{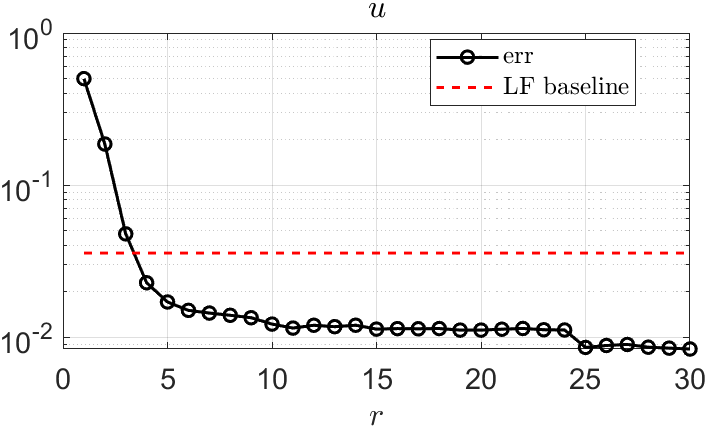}
    \includegraphics[width=0.48\linewidth]{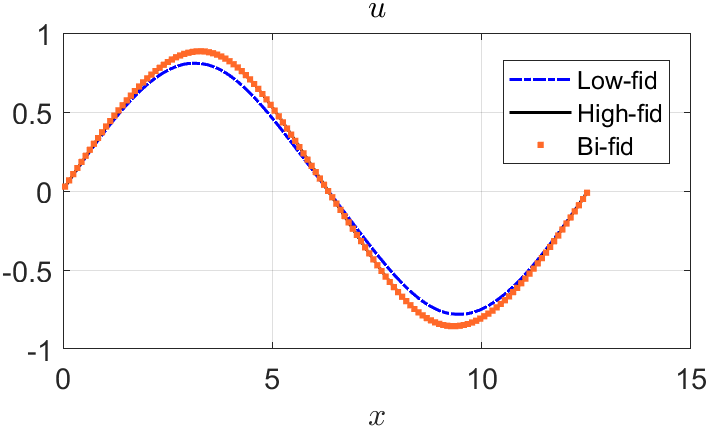}
    \caption{Average $L^2$ error relative to the number of high-fidelity runs for high-fidelity vs. low-fidelity or bi-fidelity solutions (left); comparison of low-fidelity, high-fidelity solutions, and the corresponding bi-fidelity approximations (\(r = 30\)) for a fixed sample \(z\) (right).}
    \label{up_single-component}
\end{figure}
\begin{figure}[!htbp]
    \centering
    \includegraphics[width=0.48\linewidth]{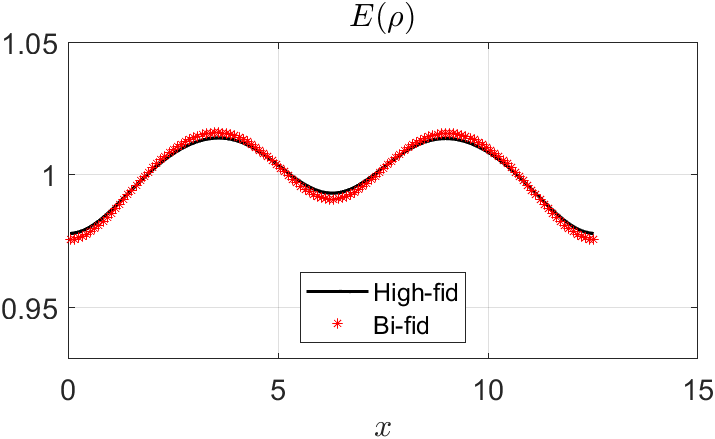}
    \includegraphics[width=0.48\linewidth]{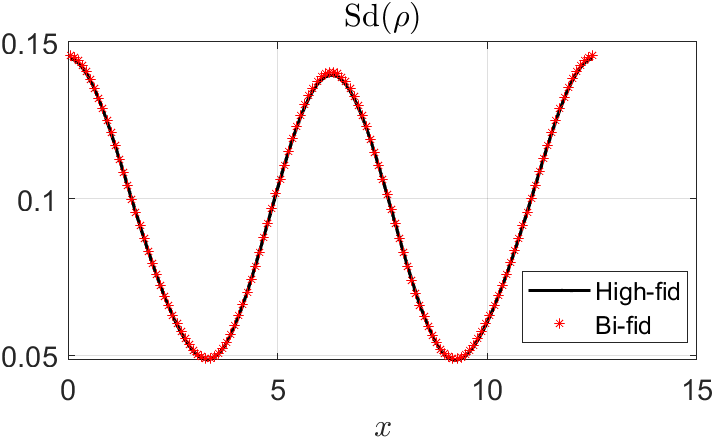}\\
    \includegraphics[width=0.50\linewidth]{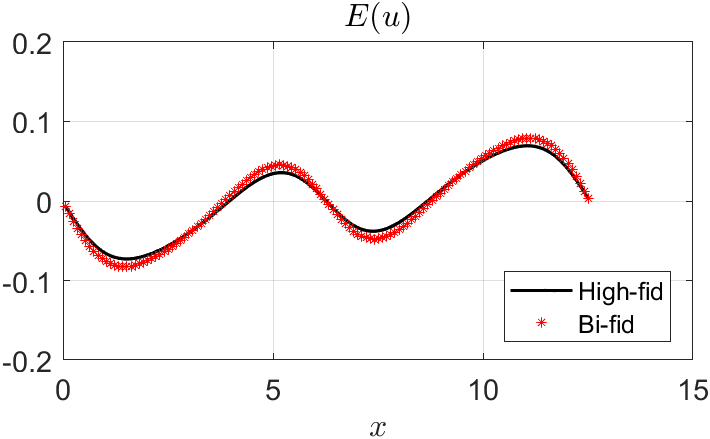}
    \includegraphics[width=0.48\linewidth]{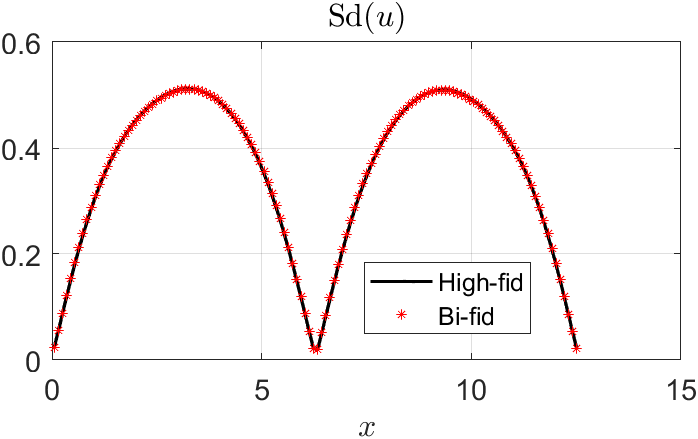}
    \caption{Mean(left) and standard deviation(right) of \(\rho\), \(u\) of high-fidelity solutions and bi-fidelity solutions with \(r = 30\).}
    \label{up_single-component_E_SD}
\end{figure}

\subsubsection{Two-Component}
In a complete description of plasma physics, the original VP model accounts for two charged species: electrons ($s=e$) and ions ($s=i$). 
In multi-species plasma simulations, capturing the disparate dynamic behaviors of different charged species—such as electrons and ions—is essential for accurate physical predictions \cite{LABRUNIE2004267,Petkaki2006AnomalousRA}. 
The co-evolution of electron and ion distribution functions self-consistently coupled through the electrostatic field governed by the Poisson equation.
The electron mass is $m_e$ and the ion mass is $m_i$, yielding a small mass ratio $\mu = m_e/m_i \ll 1$ \cite{Cheng2015}.
The two-component VP system (1D1V) with random parameters reads
\begin{align}\label{eq:two_VP}
\begin{cases}
\partial_t f_s + v\,\partial_x f_s + q_s E\,\partial_v f_s = 0, \qquad (x,v)\in\Omega_x\times\Omega_v^s,\quad s=e,i,\\[2mm]
E = -\partial_x\phi,\qquad - \partial_x^2\phi = \rho_e-\rho_i,
\end{cases}
\end{align}
where the species-indexed charge coefficients \(q_e=+1, q_i=-\mu\). 
Similarly, the corresponding pressureless two-component EP system of \eqref{eq:two_VP} can be obtained.

\begin{example}
We are considering introducing random parameters into the following two‑component plasma of initial conditions \cite{Cheng2015}, where $\mu = m_e/m_i = 1/9$ and ion temperature $T_i = 0.3$ ($\sigma_i^2 = \mu T_i = 1/30$). 
Similar to single-component initial value perturbation, we get
\begin{align}
\begin{cases}
  f_e^0(x,v,z) &= \frac{n_e}{\sqrt{2\pi}}  e^{-(v-u_e)^2/2} \Bigl[1 + S(x,z)\Bigr], \\
  f_i^0(x,v,z)   &= \frac{n_i}{\sqrt{2\pi}\,\sigma_i}e^{-(v-u_i)^2/(2\sigma_i^2)} ,\\
  \rho_e^0(x,v,z) &= n_e\Bigl[1 + S(x,z)\Bigr], \quad (\rho u)_e^0(x,v,z) = n_eu_e\Bigl[1 + S(x,z)\Bigr], \\
  \rho_i^0(x,v,z) &= n_i , \quad (\rho u)_i^0(x,v,z) =n_iu_i,
  \end{cases}
\end{align}
where $S(x,z) = \sum_{m=1}^{d_z} \frac{\alpha_0}{m} z_m \cos(m\kappa x)$ with $\kappa = 0.5, \, \alpha_0 = 0.5$. 
Let $d_z=5$, \(N_h=N_l=128\), and the random parameters are \(n_e=1+0.05z_1, \, u_e=0.01z_2, \, n_i=1+0.05z_4,\) \(u_i=0.01z_5.\)
The velocity domains are $v_e \in [-10,10]$ and $v_i \in \sqrt{\sigma_i}v_e$.
\end{example}

Fig.\ref{up_two-component} reports the convergence history and numerical solutions for the electron–ion two-species system. The left column demonstrates rapid and robust $L^2$-error convergence for $\rho_e, u_e, \rho_i, u_i$: using only $r=30$ high-fidelity runs, the bi-fidelity approximation can achieve a maximum accuracy of $\mathcal{O}(10^{-4})$, whereas the low-fidelity baseline stagnates at a consistently poor level across all variables.
The right column compares the low-, high-, and bi-fidelity solutions at a representative spatial point; the bi-fidelity results coincide with the high-fidelity ones, while the low-fidelity model exhibits deviations in both magnitude and phase, failing to capture the correct behavior.

\begin{figure}[!htbp]
    \centering
    \includegraphics[width=0.49\linewidth]{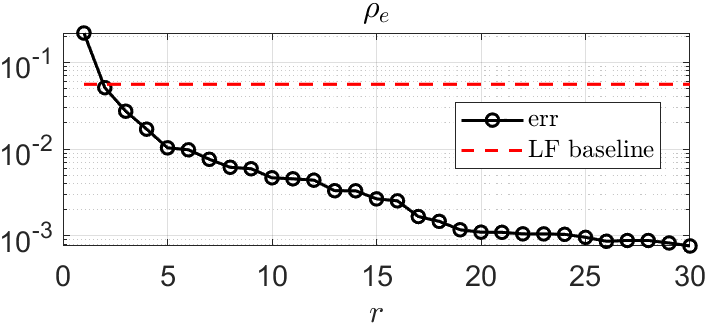}
    \includegraphics[width=0.49\linewidth]{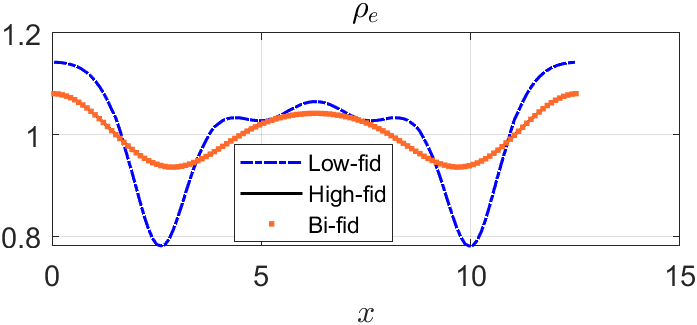}\\
    \includegraphics[width=0.49\linewidth]{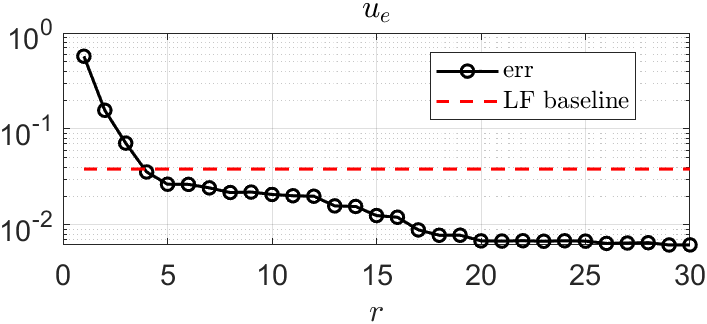}
    \includegraphics[width=0.49\linewidth]{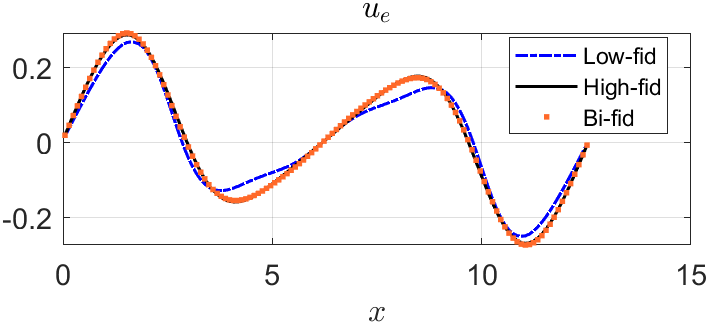}\\
    \includegraphics[width=0.49\linewidth]{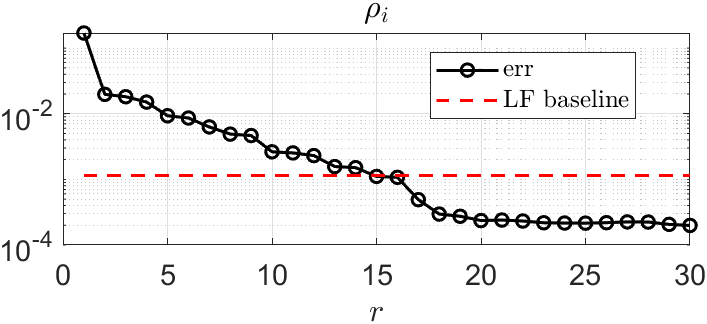}
    \includegraphics[width=0.49\linewidth]{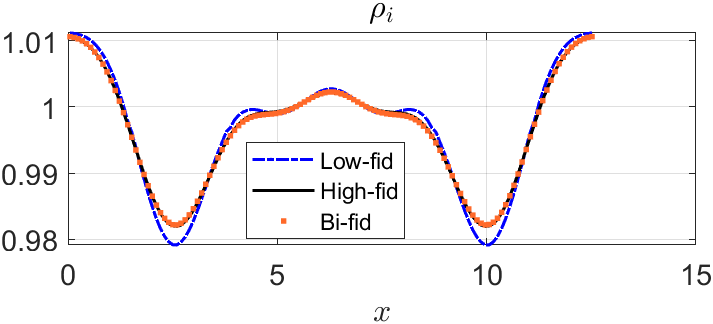}\\
    \includegraphics[width=0.49\linewidth]{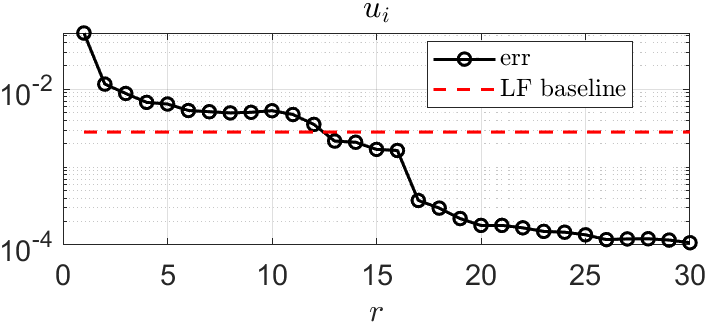}
    \includegraphics[width=0.49\linewidth]{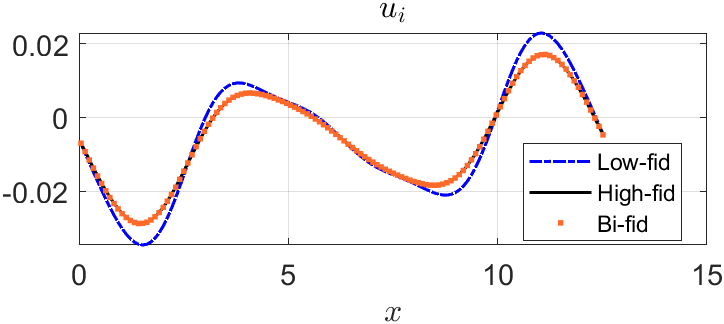}
    \caption{Two-Component: Average $L^2$ error relative to the number of high-fidelity runs for high-fidelity vs. low-fidelity or bi-fidelity solutions (left); comparison of low-fidelity, high-fidelity solutions, and the bi-fidelity solutions (\(r = 30\)) for a fixed sample \(z\) (right).}
    \label{up_two-component}
\end{figure}

Fig.\ref{up_two-component_E_SD} demonstrates the mean and standard deviation of four variables obtained from the two-component random system through the bi-fidelity approximation.
Despite the high computational cost of the high-fidelity model for this multi-component system, the bi-fidelity approach accurately captures the amplitude and spatial oscillations while achieving a significant speedup without compromising precision.

\begin{figure}[!htbp]
    \centering
    \includegraphics[width=0.49\linewidth]{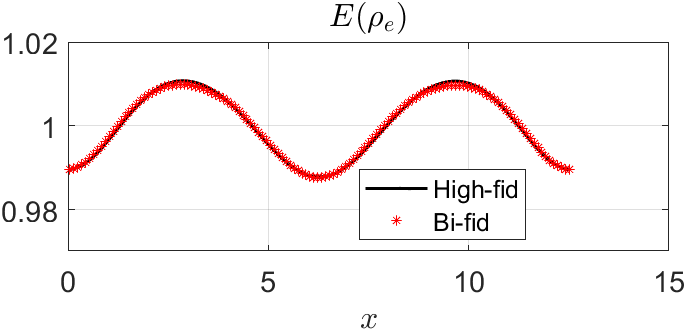}
    \includegraphics[width=0.49\linewidth]{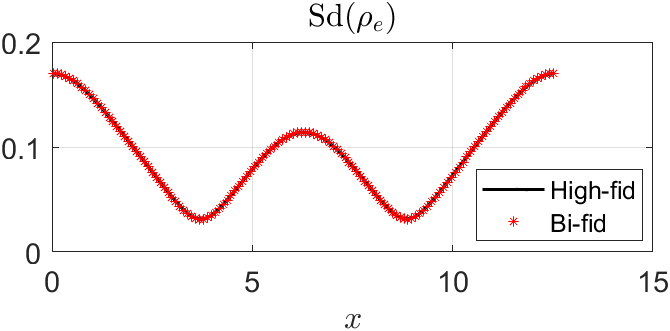}\\
    \includegraphics[width=0.49\linewidth]{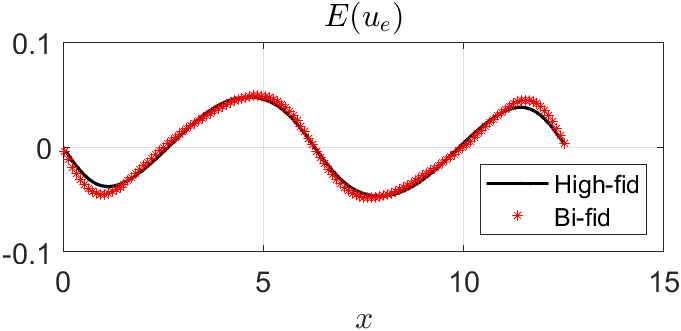}
    \includegraphics[width=0.48\linewidth]{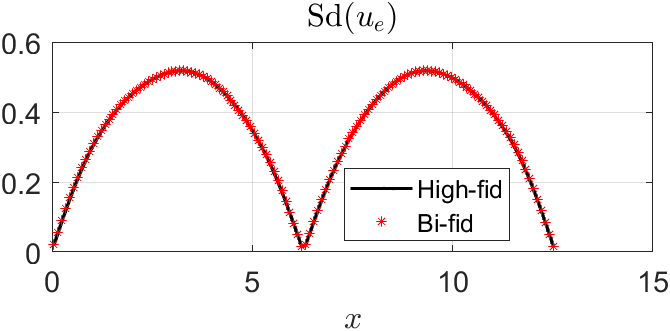}\\
    \includegraphics[width=0.49\linewidth]{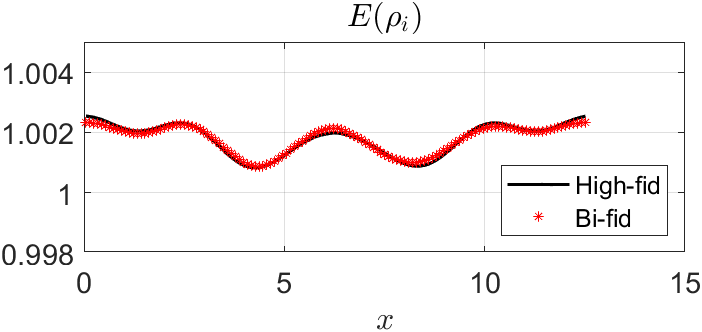}
    \includegraphics[width=0.49\linewidth]{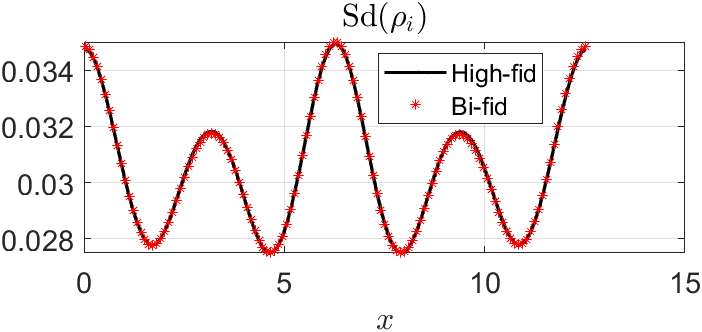}\\
    \includegraphics[width=0.48\linewidth]{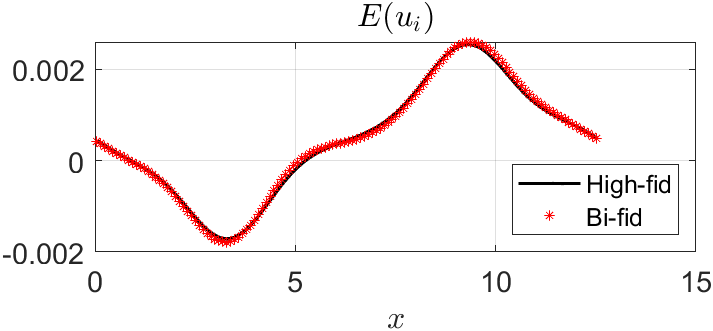}
    \includegraphics[width=0.49\linewidth]{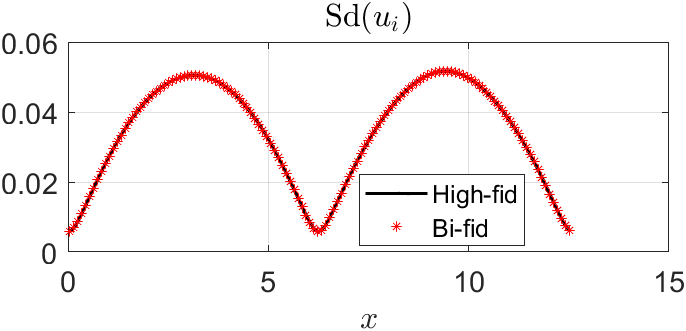}
    \caption{Two-Component: Mean(left) and standard deviation(right) of \(\rho\), \(u\) of high-fidelity solutions and bi-fidelity solutions with \(r = 30\).}
    \label{up_two-component_E_SD}
\end{figure}

\subsubsection{2D1V}
In the examples above, we have gradually verified the efficient application of the bi-fidelity algorithm in single/multi-component cases.
Next, based on the 2D1V initial value construction from the deterministic problem above, we add random parameters to the initial values that Multi‑vortex two‑stream instability in 2D1V.
\begin{example}
\begin{align}
\begin{cases}
f_0(x,y,v,z) &= \frac{1}{2v_{\text{th}}\sqrt{2\pi}}
\Bigl[\exp\!\bigl(-\tfrac{(v-u_0)^2}{2v_{\text{th}}^2}\bigr) +
\exp\!\bigl(-\tfrac{(v+u_0)^2}{2v_{\text{th}}^2}\bigr)\Bigr]
\Bigl[ 1 + S(x, z) \Bigr],\\
\rho_0(x,y,z) &= 1 + S(x, z),\\
(\rho u)_0(x,y,z) &= 0,
\end{cases}
\end{align}
where the random parameters are \(S(x, z) = \sum_{m=1}^{d_z-1} \dfrac{\alpha}{m}\, z_m \left(\cos( m \kappa_x x)+\cos( m \kappa_y y)\right),\)
\(
u_0 = 0.99 + 0.05z_{d_z}.\)
Let \(d_z=7\), \(N_h=N_l=64\), and the parameters \(v_{\text{th}}, \alpha, \kappa_x, \kappa_y\) are consistent with Example~5.2.
\end{example}

Fig.\ref{up_single-component_2D} presents the two-dimensional spatial distributions of the high-fidelity (HF), bi-fidelity (BF), and low-fidelity (LF) solutions for $\rho$ and $u$ in Example. The BF approximations accurately reproduce the complex HF spatial patterns, while the LF solutions fail to capture the essential structures. Fig.\ref{up_single-component_E_2D} further illustrate the statistical accuracy in the random space: the contour plots of the BF means and standard deviations exhibit excellent visual agreement with the HF results, with matching color scales. This confirms that the bi-fidelity method effectively captures the uncertainty propagation characteristics across the entire two-dimensional physical domain.

\begin{figure}[!htbp]
    \centering
    \includegraphics[width=0.32\linewidth]{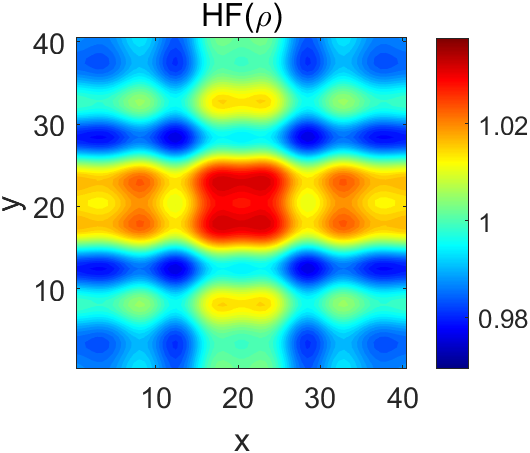}
    \includegraphics[width=0.32\linewidth]{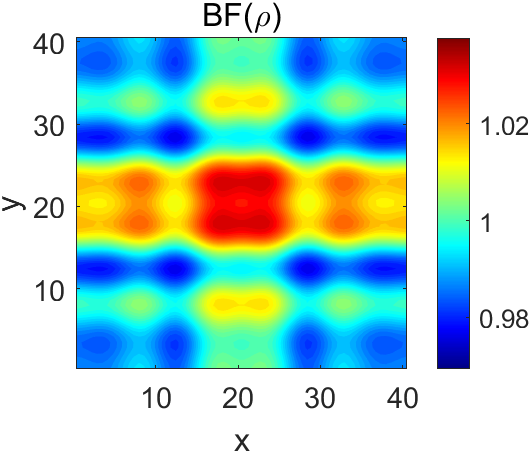}
    \includegraphics[width=0.32\linewidth]{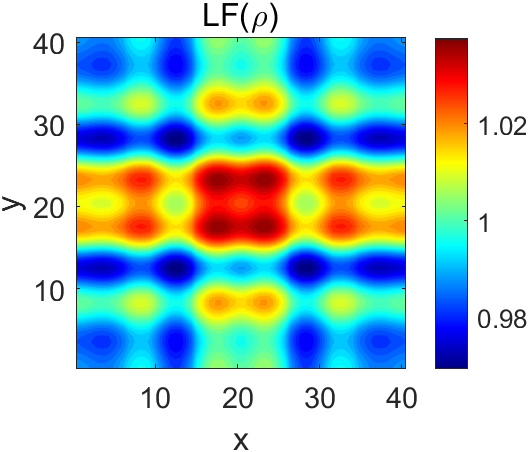}\\
    \includegraphics[width=0.32\linewidth]{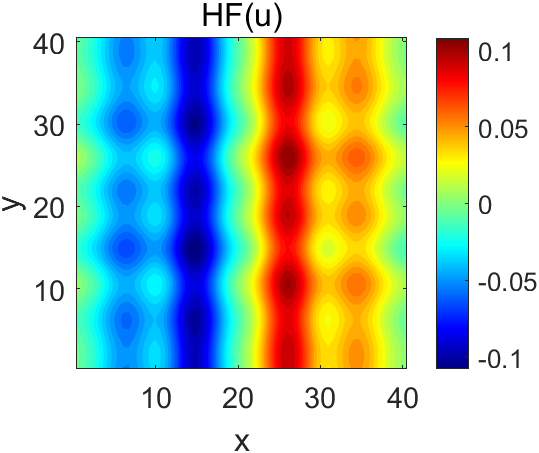}
    \includegraphics[width=0.32\linewidth]{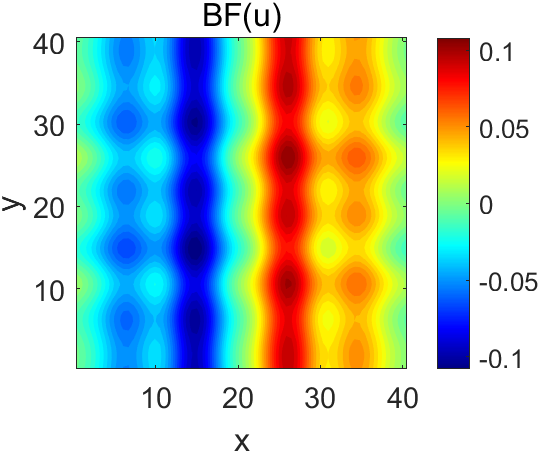}
    \includegraphics[width=0.32\linewidth]{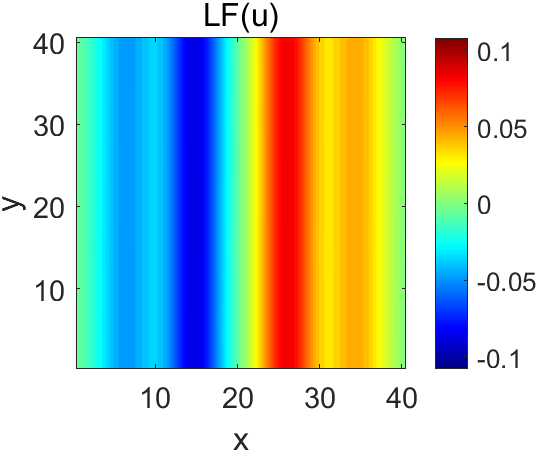}
    \caption{Comparison the two-dimensional spatial distributions of low-fidelity, high-fidelity solutions, and the corresponding bi-fidelity approximations (\(r = 30\)) for a fixed sample \(z\).}
    \label{up_single-component_2D}
\end{figure}

\begin{figure}[!htbp]
    \centering
    \includegraphics[width=0.24\linewidth]{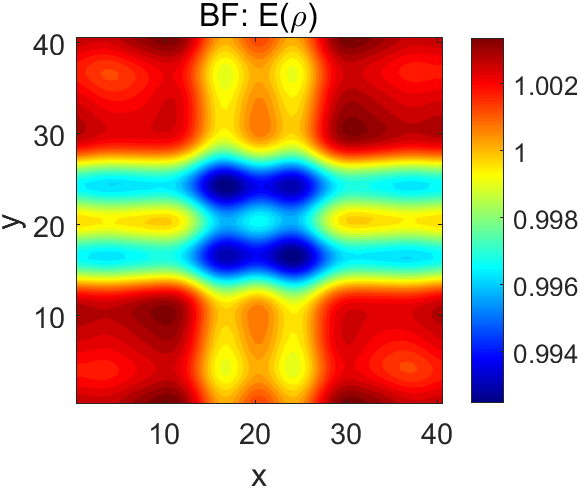}
    \includegraphics[width=0.24\linewidth]{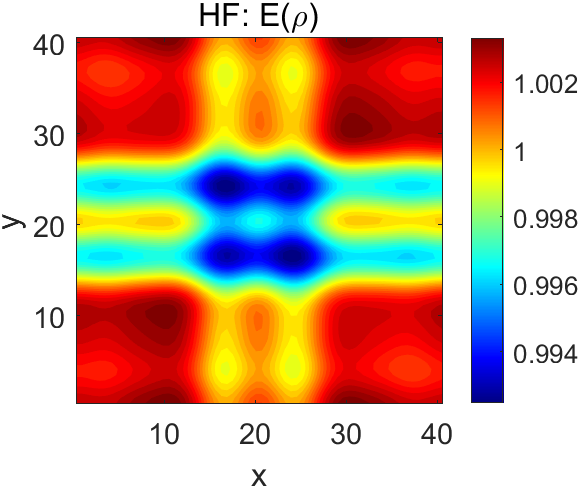}
    \includegraphics[width=0.24\linewidth]{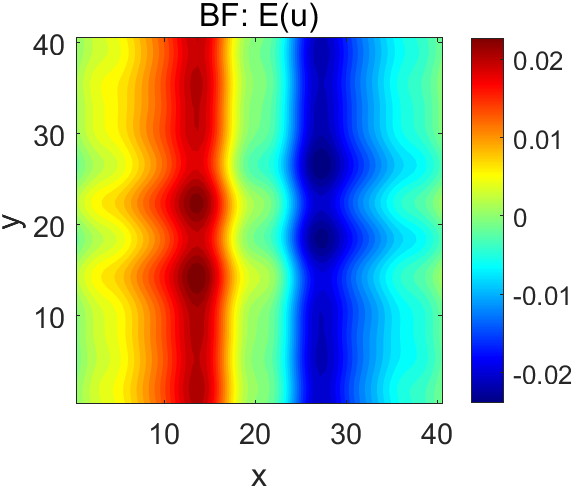}
    \includegraphics[width=0.24\linewidth]{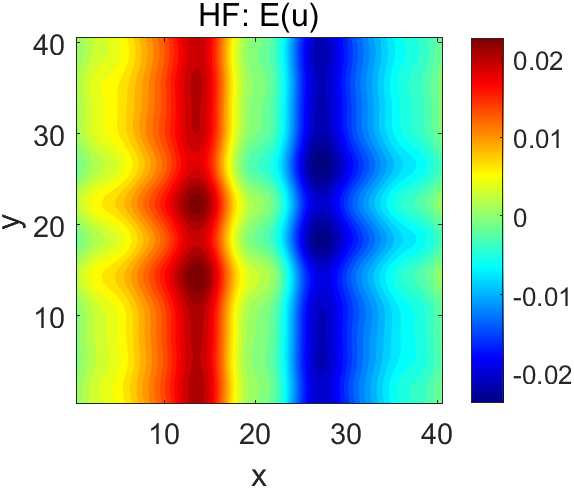}
    \includegraphics[width=0.24\linewidth]{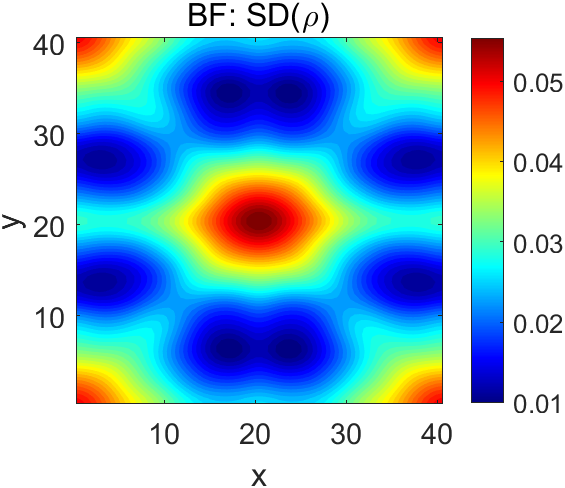}
    \includegraphics[width=0.24\linewidth]{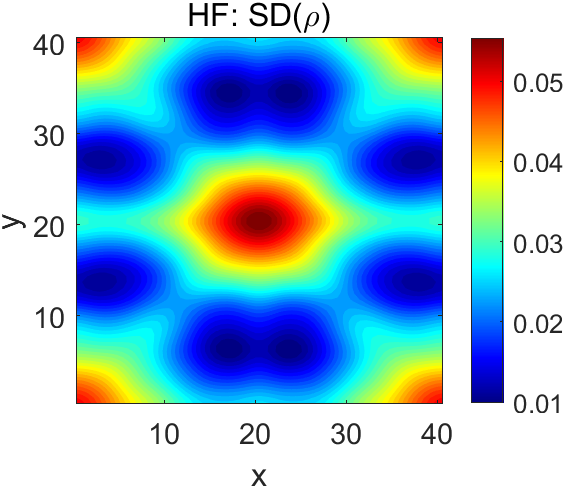}
    \includegraphics[width=0.24\linewidth]{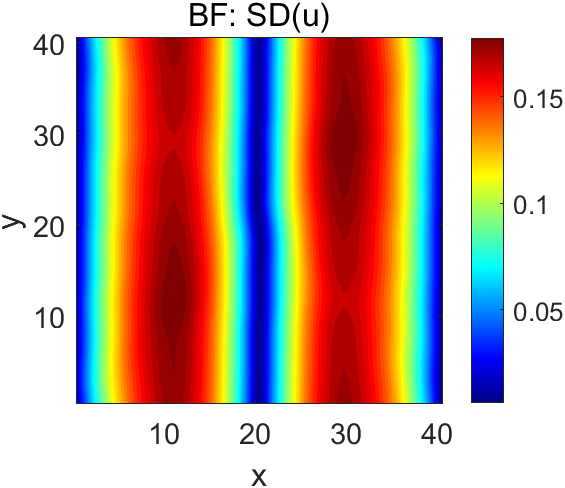}
    \includegraphics[width=0.24\linewidth]{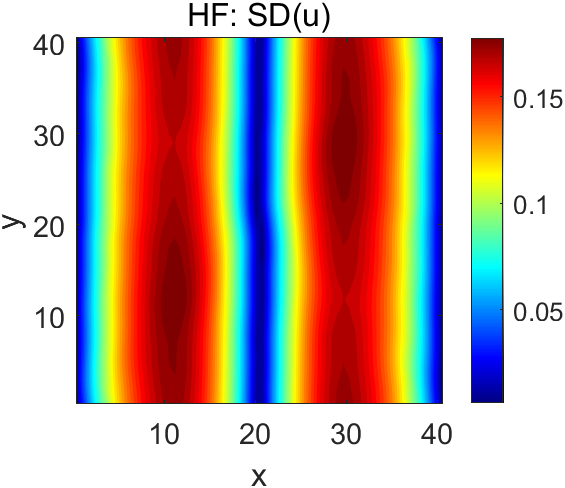}
    \caption{The two-dimensional mean(up) and standard deviation(down) of \(\rho\), \(u\) of high-fidelity solutions and bi-fidelity solutions with \(r = 30\).}
    \label{up_single-component_E_2D}
\end{figure}
\section{Conclusion}
We studied a bi‑fidelity approximation of macroscopic observables for uncertain VP problems using the conservative SL snapshots. 
The high-fidelity solver employs a Strang-splitting CSLDG scheme, which preserves structural properties (e.g., conservation and positivity), and accurately resolves fine-scale phase-space structures, thereby enhancing the bi-fidelity reconstruction of complex dynamics and uncertainty propagation.
Theoretically, we establish rigorous error estimates, decomposing the total error into several error bounds and analyzing them individually.
Numerical experiments on deterministic and uncertain problems, validate the accuracy, robustness, and computational efficiency of the proposed method. 
Future work will focus on extensions to non-splitting high-dimensional grids, multi-scale problems, or complex kinetic models.

\bibliographystyle{siamplain}
\bibliography{reference}



\end{document}


\maketitle

\section{Proof of the Lemma~4.1}
Let \((\rho_e, u_e, E_e)\) be the smooth strong solution of the pressureless EP system on the time interval \([0, T]\) \((T > 0)\), satisfying \(\rho_e \in C([0, T]; H^s(\mathbb{R}^d))\), \(u_e \in C([0, T]; H^s(\mathbb{R}^d))\), with \(s > \frac{d}{2} + 1\).
Denote \(d(u_e)(t,x) := \frac{1}{2} \left( \nabla_x u_e + (\nabla_x u_e)^T \right)\) as the symmetric strain tensor of the flow field \(u_e\), the supremum of its spectral radius in space is denoted as
\(
\| d(u_e)(t) \|_{L^\infty} = \sup_{x \in \mathbb{R}^d} \rho_{\mathrm{spec}}(d(u_e)(t,x))
\),
and assume that \(\|d(u_e)\|_{L^1([0,T]; L^\infty)} < \infty\).
Let \(f^\varepsilon \ge 0\) be a weak solution of the VP system, and suppose the total mass \(M_0 := \int_{\mathbb{R}^{2d}} f^\varepsilon(t,x,v) \, dx dv\) is conserved.

Reference \cite{Brenier01012000} has already provided the error bounds for the modulation energy \(H_\varepsilon(t)\), here we mainly revisit and summarize its important conclusions.
Following the approach in Section 3 of \cite{Brenier01012000}, the modulated total energy functional \(H_\varepsilon(t)\) for the microscopic solution \((f^\varepsilon, E^\varepsilon)\) with respect to the macroscopic strong solution \((u_e, E_e)\) is defined as:
\begin{align*}
H_\varepsilon(t) :=& \frac{1}{2} \int_{\mathbb{R}^d \times \mathbb{R}^d} |v - u_e(t,x)|^2 f^\varepsilon(t,x,v) \, dx dv + \frac{1}{2} \int_{\mathbb{R}^d} |E^\varepsilon(t,x) - E_e(t,x)|^2 \, dx\\
=& \underbrace{\frac{1}{2} \int_{\mathbb{R}^d \times \mathbb{R}^d} |v|^2 f^\varepsilon \, dx dv + \frac{1}{2} \int_{\mathbb{R}^d} |E^\varepsilon|^2 \, dx}_{\mathcal{E}_{total}(t)}+\frac{1}{2} \int_{\mathbb{R}^d} |u_e|^2 \rho^\varepsilon \, dx \\
&- \int_{\mathbb{R}^d} u_e \cdot J^\varepsilon \, dx
- \int_{\mathbb{R}^d} E^\varepsilon \cdot E_e \, dx + \frac{1}{2} \int_{\mathbb{R}^d} |E_e|^2 \, dx.
\end{align*}
Since the total energy \(\mathcal{E}_{total}(t)\) of the VP system is conserved \(\left( \frac{d}{dt} \mathcal{E}_{total}(t) = 0 \right)\), we obtain:
\begin{align*}
\frac{d}{dt} H_\varepsilon(t) = &\frac{d}{dt} \left( \frac{1}{2} \int_{\mathbb{R}^d} |u_e|^2 \rho^\varepsilon \, dx \right) - \frac{d}{dt} \left( \int_{\mathbb{R}^d} u_e \cdot J^\varepsilon \, dx \right) \\
&-  \frac{d}{dt} \left( \int_{\mathbb{R}^d} E^\varepsilon \cdot E_e \, dx \right) + \frac{1}{2} \frac{d}{dt} \left( \int_{\mathbb{R}^d} |E_e|^2 \, dx \right).
\end{align*}
Utilizing the VP mass continuity equation \(\partial_t \rho^\varepsilon = -\nabla_x \cdot J^\varepsilon\) and integration by parts:
\begin{equation}\label{err_EP_1}
\begin{split}
\frac{d}{dt} \left( \frac{1}{2} \int_{\mathbb{R}^d} |u_e|^2 \rho^\varepsilon \, dx \right) &= \int_{\mathbb{R}^d} \rho^\varepsilon u_e \cdot \partial_t u_e \, dx + \frac{1}{2} \int_{\mathbb{R}^d} |u_e|^2 (-\nabla_x \cdot J^\varepsilon) \, dx \\
&= \int_{\mathbb{R}^d} \rho^\varepsilon u_e \cdot \partial_t u_e \, dx + \int_{\mathbb{R}^d} J^\varepsilon \cdot ((u_e \cdot \nabla_x)u_e) \, dx.
\end{split}
\end{equation}
Utilizing the VP momentum conservation equation \(\partial_t J^\varepsilon = -\nabla_x \cdot \left( \int_{\mathbb{R}^d} v \otimes v f^\varepsilon \, dv \right) + \rho^\varepsilon E^\varepsilon\), and the algebraic identity \(v \otimes v = (v - u_e) \otimes (v - u_e) + v \otimes u_e + u_e \otimes v - u_e \otimes u_e\).
Using tensor contraction and the Einstein summation convention, and integration by parts:
\begin{equation}\label{err_EP_2}
\begin{split}
-\frac{d}{dt} \left( \int_{\mathbb{R}^d} u_e \cdot J^\varepsilon \, dx \right) = &-\int_{\mathbb{R}^d} \partial_t u_e \cdot J^\varepsilon \, dx  -2 \int_{\mathbb{R}^d} J^\varepsilon \cdot \left( (u_e \cdot \nabla_x)u_e\right)\, dx \\
&- \int_{\mathbb{R}^d} \nabla_x u_e : \left( \int_{\mathbb{R}^d} (v-u_e) \otimes (v-u_e) f^\varepsilon \, dv \right) \, dx \\
&+ \int_{\mathbb{R}^d} \rho^\varepsilon u_e \left( (u_e \cdot \nabla_x)u_e\right)\, dx - \int_{\mathbb{R}^d} \rho^\varepsilon u_e \cdot E^\varepsilon \, dx.
\end{split}
\end{equation}
Combining \eqref{err_EP_1} and \eqref{err_EP_2}, extracting the common factor, we get
\begin{equation}\label{err_EP_1_2}
\begin{split}
\frac{d}{dt} \left(\frac{1}{2} \int_{\mathbb{R}^d} |u_e|^2 \rho^\varepsilon \, dx \right) - & \frac{d}{dt} \left( \int_{\mathbb{R}^d} u_e \cdot J^\varepsilon \, dx \right) \\=& \int_{\mathbb{R}^d} (\rho^\varepsilon u_e - J^\varepsilon) \cdot \left[ \partial_t u_e + (u_e \cdot \nabla_x)u_e \right]\, dx - \int_{\mathbb{R}^d} \rho^\varepsilon u_e \cdot E^\varepsilon \, dx\\
&- \int_{\mathbb{R}^d} d(u_e) : \left( \int_{\mathbb{R}^d} (v-u_e) \otimes (v-u_e) f^\varepsilon \, dv \right) \, dx.
\end{split}
\end{equation}
Utilizing \(E_e = -\nabla_x \phi_e\) and \(E^\varepsilon = -\nabla_x \phi^\varepsilon\), as well as the time derivatives of the Poisson equations \( \nabla_x \cdot \partial_t E^\varepsilon = -\nabla_x \cdot J^\varepsilon\) and \( \nabla_x \cdot \partial_t E_e = -\nabla_x \cdot (\rho_e u_e)\), we have:
\begin{equation}\label{err_EP_3_4}
\begin{split}
 -  \frac{d}{dt} \left( \int_{\mathbb{R}^d} E^\varepsilon \cdot E_e \, dx \right) + \frac{1}{2} \frac{d}{dt} \left( \int_{\mathbb{R}^d} |E_e|^2 \, dx \right)= &\int_{\mathbb{R}^d} J^\varepsilon \cdot E_e \, dx + \int_{\mathbb{R}^d} \rho_e u_e \cdot E^\varepsilon \, dx \\
 &- \int_{\mathbb{R}^d} \rho_e u_e \cdot E_e \, dx.
\end{split}
\end{equation}
Now, we add \eqref{err_EP_1_2} and \eqref{err_EP_3_4} together, due to the momentum conservation equation \(\partial_t u_e + (u_e \cdot \nabla_x)u_e -E_e=0\), and recombine to obtain the complete expression:
\begin{equation}\label{err_EP_1_2_3_4}
\begin{split}
    \frac{d}{dt}H_\varepsilon(t) &= - \int_{\mathbb{R}^d} (\rho^\varepsilon - \rho_e)u_e \cdot (E^\varepsilon - E_e) \, dx \\
    &\quad - \int_{\mathbb{R}^d} d(u_e) : \left( \int_{\mathbb{R}^d} (v - u_e) \otimes (v - u_e) f^\varepsilon \, dv \right) dx.
\end{split}
\end{equation}
Let the kinetic part of the microscopic modulated energy is:
\(
\mathcal{K}_\varepsilon(t) = \frac{1}{2} \int_{\mathbb{R}^d \times \mathbb{R}^d} |v - u_e(t,x)|^2 f^\varepsilon(t,x,v) \, dx dv,
\) utilizing the property of tensor contraction \(A : (w \otimes w) \le \rho_{\mathrm{spec}}(A) |w|^2\):
\begin{align*}
\begin{split}
\left| -\int_{\mathbb{R}^d} d(u_e) : (v - u_e) \otimes (v - u_e) f^\varepsilon \, dx dv \right| &\leq \|d(u_e)\|_{L^\infty} \int_{\mathbb{R}^d \times \mathbb{R}^d} |v - u_e|^2 f^\varepsilon \, dx dv \\
&:= 2\|d(u_e)\|_{L^\infty} \mathcal{K}_\varepsilon(t).
\end{split}
\end{align*}
Similarly, let the electric field component of the microscopic modulated energy is:
\(
\mathcal{E}_\varepsilon(t) = \frac{1}{2} \int_{\mathbb{R}^d} |E^\varepsilon(t,x) - E_e(t,x)|^2 dx.
\)
Using the irrotational property of the electric field and the technique of integration by parts, we obtain:
\begin{align*}
\left|- \int_{\mathbb{R}^d} (\rho^\varepsilon - \rho_e)u_e \cdot (E^\varepsilon - E_e) \, dx\right|
\leq& 2 \|d(u_e)\|_{L^\infty} \cdot \frac{1}{2} \int_{\mathbb{R}^d} |E^\varepsilon - E_e|^2 \, dx \\
:=& 2 \|d(u_e)\|_{L^\infty} \mathcal{E}_\varepsilon(t).
\end{align*}
Combining the above two terms:
\(
    \frac{d}{dt} H_\varepsilon(t) \leq 2 \|d(u_e)(t)\|_{L^\infty} \left( \mathcal{K}_\varepsilon(t) + \mathcal{E}_\varepsilon(t) \right) \\
    = 2 \|d(u_e)(t)\|_{L^\infty} H_\varepsilon(t).
\)
Finally, applying Grönwall's inequality gets the main conclusions in Section 3 of \cite{Brenier01012000}:
\begin{align}\label{main_error}
H_\varepsilon(t) \le H_\varepsilon(0) \cdot \exp \left( \int_{0}^{t} 2 \|d(u_e)(s)\|_{L^\infty} \, ds \right).
\end{align}
To further obtain the macro-scale measurement error.
The kinetic part of the microscopic modulated energy,
applying the Cauchy-Schwarz inequality in velocity space,
and ubstituting the moment definitions \(J^\varepsilon = \int_{\mathbb{R}^d} v f^\varepsilon dv\) and \(\rho^\varepsilon = \int_{\mathbb{R}^d} f^\varepsilon dv\):
\[
|J^\varepsilon(t,x) - \rho^\varepsilon(t,x)u_e(t,x)|^2 \le \rho^\varepsilon(t,x) \int_{\mathbb{R}^d} |v - u_e(t,x)|^2 f^\varepsilon(t,x,v) \, dv.
\]
In the non-vacuum region \(\{ \rho^\varepsilon(t,x) > 0 \}\), substituting the microscopic mean velocity \(u^\varepsilon = J^\varepsilon / \rho^\varepsilon\), and integrating over physical space \(x \in \mathbb{R}^d\), we obtain:
\begin{align}\label{err_rho_u}
    \int_{\mathbb{R}^d} \rho^\varepsilon(t,x)|u^\varepsilon(t,x) - u_e(t,x)|^2 \, dx \le 2 \mathcal{K}_\varepsilon(t).
\end{align}
Similarly, for the electric field component of the microscopic modulated energy,
subtracting the Poisson equations of the two systems yields:
\(
- \Delta_x (\phi^\varepsilon - \phi_e) = \rho^\varepsilon - \rho_e \implies \phi^\varepsilon - \phi_e = (-\Delta_x)^{-1} \left( \rho^\varepsilon - \rho_e\right).
\)
Since \(E^\varepsilon - E_e = -\nabla_x(\phi^\varepsilon - \phi_e)\), utilizing Green's identity \(\int_{\mathbb{R}^d}|\nabla\phi|^2 dx=\int_{\mathbb{R}^d} \phi (-\Delta\phi) dx\), we get the
\(
\int_{\mathbb{R}^d} |E^\varepsilon - E_e|^2 dx = \int_{\mathbb{R}^d} |\nabla_x (\phi^\varepsilon - \phi_e)|^2 dx
= \left\| \rho^\varepsilon - \rho_e \right\|_{{H}^{-1}(\mathbb{R}^d)}^2.
\)
Now, we obtain:
\begin{align}\label{err_rho}
\mathcal{E}_\varepsilon(t) = \frac{1}{2} \int_{\mathbb{R}^d} |E^\varepsilon - E_e|^2 dx = \frac{1}{2} \left\| \rho^\varepsilon(t) - \rho_e(t) \right\|_{{H}^{-1}}.
\end{align}
Then, adding \eqref{err_rho_u} and \eqref{err_rho} yields:
\begin{align*}
   &\frac{1}{2} \mathcal{D}_{macro}^2 \left( (\rho^\varepsilon(t), u^\varepsilon(t)), (\rho_e(t), u_e(t)) \right)\\
   &= \frac{1}{2} \int_{\mathbb{R}^d} \rho^\varepsilon(t,x)|u^\varepsilon(t,x) - u_e(t,x)|^2 \, dx + \frac{1}{2} \left\| \rho^\varepsilon(t) - \rho_e(t) \right\|_{{H}^{-1}}\\
   &\le \mathcal{K}_\varepsilon(t) + \mathcal{E}_\varepsilon(t) =H_\varepsilon(t).
\end{align*}
According to the main theorem derived within the framework of \cite{Brenier01012000}, see above \eqref{main_error}), this leads to the core lower-bound inequality:
\[
\mathcal{D}_{macro}^2 \left( (\rho^\varepsilon(t), u^\varepsilon(t)), (\rho_e(t), u_e(t)) \right)
\le 2 H_\varepsilon(0) \exp \left( 2 \int_0^t \left\| d(u_e)(s) \right\|_{L^\infty(\mathbb{R}^d)} ds \right)\le C'.
\]
This completes the rigorous proof of the Lemma~4.1.

\section{Proof of the Lemma~4.2}
\label{app:error_proof_VP}
The rigorous convergence analysis of the Strang splitting algorithm for the spatially discontinuous Galerkin approximation of the VP equation originates from reference \cite{Einkemmer2014}. Here, we mainly review the important conclusions and outline the proof ideas.

Let the numerical solution be given by \(\tilde{f}_{n+1} = \tilde{S}^{(\mathcal{A})} P e^{\tau \mathcal{B}(\tilde{f}_{n+1/2})} \tilde{S}^{(\mathcal{A})} \tilde{f}_n\), and the exact solution be \(f_{n+1} = f(n\tau + \tau)\). Our goal is to decompose the error \(e_{n+1} = \|\tilde{f}_{n+1} - f_{n+1}\|_2\).
To assist the derivation, we introduce intermediate state  \(\tilde{f}_{n+1/2} = P e^{\frac{\tau}{2}\mathcal{A}}\tilde{f}_n\), and reference state \(f_{n+1/2}^{ref} = P e^{\frac{\tau}{2} \mathcal{A}} P f_n\). Utilizing the triangle inequality, we can split \(e_{n+1}\) by introducing intermediate terms through a four-step process:

\textbf{Step 1:}
We replace \(f_{n+1}\) with \(P f_{n+1} + (I - P)f_{n+1}\), which yields:
\begin{align*}
e_{n+1} &= \|\tilde{f}_{n+1} - f_{n+1}\|_2 \\
&= \left\| \big(\tilde{f}_{n+1} - P f_{n+1}\big) + \big(P f_{n+1} - f_{n+1}\big) \right\|_2 \\
&\le \|\tilde{f}_{n+1} - P f_{n+1}\|_2 + \|P f_{n+1} - f_{n+1}\|_2.
\end{align*}
Notice that the second term in the inequality above is the projection error of the smooth solution at the current time, and its error bound can be derived from Theorem 2.3 in \cite{Einkemmer2014} regarding the classical DG projection error: \(\|P f_{n+1} - f_{n+1}\|_2\le C h^{p+1}\).
Immediately following, We denote the remaining part as \(R_0 = \|\tilde{f}_{n+1} - P f_{n+1}\|_2\) and proceed to further decompose \(R_0\).

\textbf{Step 2:}
We introduce the reference term \(\tilde{S}^{(\mathcal{A})} P e^{\tau \mathcal{B}(f_{n+1/2}^{ref})} \tilde{S}^{(\mathcal{A})} \tilde{f}_n\) (i.e., we assume the operator \(\mathcal{B}\) uses the reference state):
\begin{align*}
R_0 &= \|\tilde{f}_{n+1} - P f_{n+1}\|_2 \\
&= \left\| \tilde{S}^{(\mathcal{A})} P e^{\tau \mathcal{B}(\tilde{f}_{n+1/2})} \tilde{S}^{(\mathcal{A})} \tilde{f}_n - P f_{n+1} \right\|_2 \\
&\le \underbrace{\left\| \tilde{S}^{(\mathcal{A})} P \left(e^{\tau \mathcal{B}(\tilde{f}_{n+1/2})} - e^{\tau \mathcal{B}(f_{n+1/2}^{ref})}\right) \tilde{S}^{(\mathcal{A})} \tilde{f}_n \right\|_2}_{\text{Term 1: Nonlinear operator state error}}\\
&+ \left\| \tilde{S}^{(\mathcal{A})} P e^{\tau \mathcal{B}(f_{n+1/2}^{ref})} \tilde{S}^{(\mathcal{A})} \tilde{f}_n - P f_{n+1} \right\|_2,
\end{align*}
where the sensitivity of the nonlinear evolution operator \( S^{(\mathcal{B})} \) with respect to the input state is Lipschitz continuous over the time interval \( \tau \). Combined with the uniformly boundedness assumption on the electric field \( E \), and the definitions of the intermediate state. Using the nonexpansiveness of the projection operator, and using the constant variation formula (Lemma~2.5 in reference  \cite{Einkemmer2014}), we obtain:
\begin{align*}
\text{Term 1} &\leq C\tau \|\tilde{f}_{n+1/2} - f_{n+1/2}^{\text{ref}}\|_2\\
&\leq C\tau \|P e^{\frac{\tau}{2}\mathcal{A}}\tilde{f}_n - P e^{\frac{\tau}{2}\mathcal{A}} f_n\|_2
     + \|P e^{\frac{\tau}{2}\mathcal{A}} f_n - P e^{\frac{\tau}{2}\mathcal{A}} P f_n\|_2 \\
&\leq C\tau \|\tilde{f}_n - f_n\|_2 + \|P e^{\frac{\tau}{2}\mathcal{A}} (I - P) f_n\|_2\leq C\tau e_n + C\tau h^{p+1}.
\end{align*}

\textbf{Step 3:}
Let \(R_1 = \left\| \tilde{S}^{(\mathcal{A})} P e^{\tau \mathcal{B}(f_{n+1/2}^{ref})} \tilde{S}^{(\mathcal{A})} \tilde{f}_n - P f_{n+1} \right\|_2\), we introduce the reference input term \(\tilde{S}^{(\mathcal{A})} P e^{\tau \mathcal{B}(f_{n+1/2}^{ref})} \tilde{S}^{(\mathcal{A})} P f_n\):
\begin{align*}
R_1 &= \left\| \tilde{S}^{(\mathcal{A})} P e^{\tau \mathcal{B}(f_{n+1/2}^{ref})} \tilde{S}^{(\mathcal{A})} \tilde{f}_n - P f_{n+1} \right\|_2 \\
&\le \underbrace{\left\| \tilde{S}^{(\mathcal{A})} P e^{\tau \mathcal{B}(f_{n+1/2}^{ref})} \tilde{S}^{(\mathcal{A})} \big(\tilde{f}_n - P f_n\big) \right\|_2}_{\text{Term 2: Propagation of current global error}} + \left\| \tilde{S}^{(\mathcal{A})} P e^{\tau \mathcal{B}(f_{n+1/2}^{ref})} \tilde{S}^{(\mathcal{A})} P f_n - P f_{n+1} \right\|_2.
\end{align*}
Term 2 accounts for the linear propagation of the pre-existing error from the current time step to the next.
\(P\) is an orthogonal projection, and the operator norms of the translation operators are all \(1\) (in the \(L^2\) setting). Hence the overall operator norm is no greater than \(1\):
\[
\text{Term 2} \leq \|\tilde{f}_n - P f_n\|_2 \leq \|\tilde{f}_n - f_n\|_2 + \|f_n - P f_n\|_2\leq e_n + C h^{p+1}.
\]

\textbf{Step 4:}
Let  \(R_2 = \left\| \tilde{S}^{(\mathcal{A})} P e^{\tau \mathcal{B}(f_{n+1/2}^{ref})} \tilde{S}^{(\mathcal{A})} P f_n - P f_{n+1} \right\|_2\), we introduce the reference input term \(\tilde{S}^{(\mathcal{A})} P e^{\tau \mathcal{B}(f_{n+1/2}^{ref})} \tilde{S}^{(\mathcal{A})} f_n\):
\begin{align*}
R_2 &= \left\| \tilde{S}^{(\mathcal{A})} P e^{\tau \mathcal{B}(f_{n+1/2}^{ref})} \tilde{S}^{(\mathcal{A})} P f_n - P f_{n+1} \right\|_2 \\
&\le \underbrace{\left\| \tilde{S}^{(\mathcal{A})} P e^{\tau \mathcal{B}(f_{n+1/2}^{ref})} \tilde{S}^{(\mathcal{A})} \big(P f_n - f_n\big) \right\|_2}_{\text{Term 3: Propagation of last-step projection error}} + \underbrace{\left\| \tilde{S}^{(\mathcal{A})} P e^{\tau \mathcal{B}(f_{n+1/2}^{ref})} \tilde{S}^{(\mathcal{A})} f_n - P f_{n+1} \right\|_2}_{\text{Term 4: Single-step consistency error}}
\end{align*}

 Term 3 is the propagation of the initial projection error \(P f_n - f_n\) from the previous time step. Since  Theorem 2.3 in \cite{Einkemmer2014} and the non-expansiveness of an operator, this term is directly bounded: \(\text{Term 3}\le C h^{p+1}.\)

 Term 4 corresponds to the fully discrete single-step error triggered by the pure initial exact state \(f_n\), which precisely corresponds to Theorem 2.8 (Local Consistency Theorem) in \cite{Einkemmer2014, Einkemmer201452}:  \(\text{Term 4}\le C\left(\tau^3 + \tau h^{p+1} + h^{p+1}\right).\)

\textbf{Final combination:}
By combining the five components of the error generated from Step 1 to Step 4, we arrive at the final global error decomposition
we extract and merge similar terms, the error inequality can be written as:
\begin{align*}
e_{n+1} &= \underbrace{\left\| \tilde{S}^{(\mathcal{A})} P \left(e^{\tau \mathcal{B}(\tilde{f}_{n+1/2})} - e^{\tau \mathcal{B}(f_{n+1/2}^{ref})}\right) \tilde{S}^{(\mathcal{A})} \tilde{f}_n \right\|_2}_{\text{Term 1: Nonlinear operator error}}\\
&\quad +\underbrace{\left\| \tilde{S}^{(\mathcal{A})} P e^{\tau \mathcal{B}(f_{n+1/2}^{ref})} \tilde{S}^{(\mathcal{A})} \big(\tilde{f}_n - P f_n\big) \right\|_2}_{\text{Term 2: Stability propagation of current error}} \\
&\quad + \underbrace{\left\| \tilde{S}^{(\mathcal{A})} P e^{\tau \mathcal{B}(f_{n+1/2}^{ref})} \tilde{S}^{(\mathcal{A})} \big(P f_n - f_n\big) \right\|_2}_{\text{Term 3: Propagation of last-step projection error}}\\
&\quad + \underbrace{\left\| \tilde{S}^{(\mathcal{A})} P e^{\tau \mathcal{B}(f_{n+1/2}^{ref})} \tilde{S}^{(\mathcal{A})} f_n - P f(n\tau + \tau) \right\|_2}_{\text{Term 4: Single-step full-discrete consistency error }} \\
&\quad + \underbrace{\left\| P f(n\tau + \tau) - f(n\tau + \tau) \right\|_2}_{\text{Term 5: Final smooth projection error}}\\
&\le (1 + C\tau)e_n + C\tau h^{p+1} + Ch^{p+1} + Ch^{p+1} + C\tau^3 + C\tau h^{p+1} + Ch^{p+1} + Ch^{p+1} \\
&\le (1 + C\tau)e_n + C\left(\tau^3 + \tau h^{p+1} + h^{p+1}\right).
\end{align*}

Then, assume the total simulation time is \(T = N_t\tau\), and the initial error satisfies \(e_0 = \|P f_0 - f_0\| \le Ch^{p+1}\). By expanding the above recurrence inequality, we have:
\begin{align*}
e_N &\le (1 + C\tau)^N e_0 + \sum_{k=0}^{N-1} (1 + C\tau)^{N-1-k} C\left(\tau^3 + \tau h^{p+1} + h^{p+1}\right) \\
&\le e^{CT} e_0 + N \cdot e^{CT} \cdot C\left(\tau^3 + \tau h^{p+1} + h^{p+1}\right).
\end{align*}
Substituting \(N_t = \frac{T}{\tau}\), and recognizing that the constants \(C, T, e^{CT}\) are all bounded constants (which can be uniformly denoted as \(C\)), and using the Discrete Gr\"onwall Lemma, we obtain:
\begin{equation*}
\sup_{0 \le n \le N} e_n \le Ch^{p+1} + \frac{T}{\tau} \cdot C\left(\tau^3 + \tau h^{p+1} + h^{p+1}\right)
\le C \left( \tau^2 + h^{p+1} + \frac{h^{p+1}}{\tau} \right).
\end{equation*}

\bibliographystyle{siamplain}
\bibliography{references}